\documentclass[12pt]{article}

\usepackage{latexsym}
\usepackage{a4wide}
\usepackage{amscd}
\usepackage{graphics}
\usepackage{amsmath}
\usepackage{amssymb}
\usepackage{mathrsfs}
\usepackage{amsthm}
\usepackage{bbm}
\usepackage{color}
\usepackage{accents}
\usepackage{enumerate}
\usepackage{lipsum}
\input xy
\xyoption{all}

\newcommand{\N}{\mathbb{N}}                     
\newcommand{\Z}{\mathbb{Z}}                     
\newcommand{\R}{\mathbb{R}}                     
\newcommand{\C}{\mathbb{C}}                     
\newcommand{\T}{\mathbb{T}}                     
\newcommand{\set}[2]{\left\{{#1}\mid{#2}\right\}}       
\newcommand{\im}{\mathrm{Im\,}}                 
\newcommand{\re}{\mathrm{Re\,}}                 
\newcommand{\A}{\mathbb{A}}
\newcommand\blfootnote[1]{%
  \begingroup
  \renewcommand\thefootnote{}\footnote{#1}%
  \addtocounter{footnote}{-1}%
  \endgroup
}

\newtheorem{mainthm}{\sc Theorem}           
\newtheorem{maincor}[mainthm] {\sc Corollary}          
\newtheorem{thm}{\sc Theorem}[section]               
\newtheorem*{thm*}{\sc Theorem}               
\newtheorem{cor}[thm]{\sc Corollary}        
\newtheorem*{cor*}{\sc Corollary}        
\newtheorem{lem}[thm]{\sc Lemma}            
\newtheorem{prop}[thm]{\sc Proposition}     
\newtheorem{rem}[thm]{\sc Remark}           
\newtheorem{ex}[thm]{\sc Example}           

\title{A non-squeezing theorem for convex symplectic images of the Hilbert ball}

\author{Alberto Abbondandolo and Pietro Majer\blfootnote{The first author is partially supported by the DFG grant AB 360/1-1.  The present work is part of the first author's activities within CAST, a Research Network Program of the European Science Foundation.}} 

\date{August 18, 2014}

\begin{document}

\maketitle

\begin{abstract}
We prove that the non-squeezing theorem of Gromov holds for symplectomorphisms on an infinite-dimensional symplectic Hilbert space, under the assumption that the image of the ball is convex. The proof is based on the construction by duality methods of a symplectic capacity for bounded convex neighbourhoods of the origin. We also discuss some examples of symplectomorphisms on infinite-dimensional spaces exhibiting behaviours which would be impossible in finite dimensions.

\tableofcontents

\end{abstract}

\newpage

\section*{Introduction}
\addcontentsline{toc}{section}{\numberline{}Introduction}

The aim of this paper is to extend the non-squeezing theorem of Gromov to symplectomorphisms of an infinite-dimensional Hilbert space, under the assumption that the image of the ball is convex. Before giving a precise statement, we review Gromov's finite-dimensional statement and its subsequent generalizations to Hamiltonian PDEs.

Let $\omega = \sum_{j=1}^n dp_j \wedge dq_j$ be the standard symplectic form on $\R^{2n}$. A symplectomorphism between open subsets of $\R^{2n}$ is a diffeomorphism which preserves $\omega$. The standard examples of symplectomorphisms are given by the flow of a (possibly time-dependent) Hamiltonian system
\[
\dot{q} = \partial_p H_t(q,p), \qquad \dot{p} = -\partial_q H_t(q,p). 
\]
The coordinate-free way of writing the above system is
\[
\dot{x} = X_{H_t} (x),
\]
where the Hamiltonian vector field $X_H$ is defined by the identity
\begin{equation}
\label{cfham}
\omega(X_H,\cdot) = - dH.
\end{equation} 

The non-squeezing theorem of Gromov \cite{gro85} states that if $0<s<r$, then no symplectomorphism can map a ball $B_r$ of radius $r$ into the cylinder of radius $s$
\[
Z_s := \set{(q_1,p_1,\dots,q_n,p_n)\in \R^{2n}}{q_1^2+p_1^2 < s^2}.
\]
A coordinate-free reformulation of this theorem is the following. Let $\mathbb{H}_0\subset \R^{2n}$ be a symplectic 2-plane, i.e.\ a 2-plane on which $\omega$ does not vanish, and let $P$ be the symplectic projector onto $\mathbb{H}_0$, i.e.\ the projector along its symplectic orthogonal complement. Then every symplectomorphism $\varphi: B_r \rightarrow \varphi(B_r) \subset \R^{2n}$ satisfies
\[
\mathrm{area}_{\omega} \bigl( P \varphi(B_r) \bigr) \geq \pi r^2,
\]
where the area on $\mathbb{H}_0$ is induced by the restriction of $\omega$. Indeed, the latter statement clearly implies Gromov's original formulation. On the other hand, if the $\omega$-area of $P\varphi(B_r)$ is smaller than $\pi r^2$, then this set can be mapped into a subset of a disc in $\mathbb{H}_0$ of radius $s<r$ by an area-preserving diffeomorphisms $\psi$, by a particular case of  Dacorogna and Moser's theorem \cite{dm90}, and the symplectomorphism $(\psi \times \mathrm{id}_{\mathbb{H}_0^{\perp_{\omega}}}) \circ \varphi$ would map $B_r$ into $Z_s$.

It is a long standing open question whether the non-squeezing theorem generalizes to infinite-dimensional symplectic Hilbert spaces. A symplectic form on a real Hilbert space $\mathbb{H}$ is a skew-symmetric continuous 2-form
\[
\omega: \mathbb{H} \times \mathbb{H} \rightarrow \R
\]
which is non-degenerate, in the sense that the associated linear mapping 
\[
\Omega: \mathbb{H} \rightarrow \mathbb{H}^*
\]
is an isomorphism. This notion goes back at least to the book of Chernoff and Marsden \cite{cm74}, where a skew-symmetric continuous 2-form which is non-degenerate in the above sense is also called a {\em strong} symplectic form. When $\Omega$ is just injective, the form $\omega$ is called a {\em weak} symplectic form. See \cite{kuk00} and \cite{bbz13} for an extensive discussion of the two notions. In this paper by symplectic form we always mean a strong symplectic form. 

Given a symplectic form $\omega$ on $\mathbb{H}$, there always exists an equivalent Hilbert product $(\cdot,\cdot)$ on $\mathbb{H}$ such that $\Omega$ is an isometry, or equivalently such that the bounded operator $J : \mathbb{H} \rightarrow \mathbb{H}$ which is defined by
\[
(Jx,y) = \omega(x,y) \qquad \forall x,y\in \mathbb{H},
\]
is a complex structure on $\mathbb{H}$, i.e.\ it satisfies $J^2=-I$. This Hilbert product and the induced norm $\|\cdot\|$ are said to be {\em compatible} with $\omega$. A {\em symplectomorphism} between open subsets of $\mathbb{H}$ is a diffeomorphism which preserves $\omega$.

The first investigations on the validity of the non-squeezing theorem on infinite dimensional Hilbert spaces are due to Kuksin \cite{kuk95,kuk95b} and are motivated by the implications that such a statement has for the global behavior of Hamiltonian PDEs. It has been known for a long time that many conservative evolutionary PDEs  can be thought as infinite dimensional Hamiltonian systems. To have a concrete example in mind, consider the periodic nonlinear Schr\"odinger equation
\begin{equation}
\label{nls}
-i \partial_t u + \Delta u = f(|u|)u, \qquad u=u(t,x) \in \C, \; t\in \R, \; x\in \T^n,
\end{equation}
where $\T:= \R/\Z$ and $f$ is a smooth real function. This equation can be considered as the Hamiltonian equation which is induced by the Hamiltonian function
\begin{equation}
\label{ham}
H(u) := \int_{\T^n} \left( \frac{1}{2} |\nabla u|^2 + F(|u|) \right)\, dx,
\end{equation}
where $F'(s)=sf(s)$, and by the symplectic form
\begin{equation}
\label{sympnls}
\omega(u,v) := - \im \int_{\T^n} u(x) \overline{v}(x)\, dx.
\end{equation}
This means that the equation  (\ref{nls}) can be written in the form
\[
\partial_t u = X_H(u),
\]
where the ``Hamiltonian vector field'' $X_H$ is formally defined by inserting (\ref{ham}) and (\ref{sympnls}) into the identity (\ref{cfham}). The skew-symmetric 2-form (\ref{sympnls}) is a symplectic form on $L^2(\T^n,\C)$, seen as a real Hilbert space, and the standard $L^2$-norm is compatible with $\omega$. The space $L^2(\T^n,\C)$ is sometimes called the {\em Darboux phase space} of the equation (\ref{nls}).

The same form is a weak symplectic form on the higher order Sobolev spaces $H^s(\T^n,\C)$, $s>0$. When the nonlinearity $f$ satisfies suitable growth and regularity assumptions, the Hamiltonian $H$ is a smooth functional on $H^1(\T^n,\C)$, which is therefore called the {\em energy phase space} of the equation (\ref{nls}). The fact that the energy phase space is strictly smaller than the Darboux phase space is a common feature of Hamiltonian PDEs: When the Hamiltonian is differentiable on a space $\mathbb{H}$ where the symplectic form is strong, then by the non-degeneracy of $\omega$ the identity (\ref{cfham}) defines a true vector field $X_H: \mathbb{H} \rightarrow \mathbb{H}$, and the Hamiltonian equation is an ODE on $\mathbb{H}$.  

In some cases, a Hamiltonian PDE defines a flow $\phi_t$ on its Darboux phase space $(\mathbb{H},\omega)$, and in this case each $\phi_t$ is a symplectomorphism (but the curve $\R \rightarrow \mathbb{H}$, $t\mapsto \phi_t(u_0)$, is not differentiable for a general $u_0\in \mathbb{H}$, unless we are dealing with an ODE on $\mathbb{H}$).  For instance, when $n=1$ and $F(s)=|s|^p$ with $|p|\leq 4$, the equation (\ref{nls}) defines a flow on the Darboux phase space $(L^2(\T,\C),\omega)$, where $\omega$ is defined by (\ref{sympnls}) (see \cite{bou93}).  In this case, it makes sense to ask whether the non-squeezing theorem hold: If $P$ is a symplectic projector onto a symplectic 2-plane, one wishes to know whether the projection by $P$ of the evolution of the ball of radius $r$ centered at $u_0$ is forced to have large area, or more precisely if
\[
\mathrm{area}_{\omega}\bigl( P \phi_t(B_r(u_0)) \bigr) \geq \pi r^2.
\]
Here it is important that the norm which defines $B_r(u_0)$ is compatible with $\omega$. In the case of the nonlinear Schr\"odinger equation, one may project on the complex line given by the $k$-th Fourier coefficient, $k\in \Z^n$,
\[
(P u)(x) := \hat{u}(k) e^{2\pi i k \cdot x}, \qquad \mbox{where} \qquad u(x) = \sum_{h\in \Z^n} \hat{u}(h) e^{2\pi i h\cdot x} , \qquad \hat{u}(h)\in \C,
\]
and the question becomes whether the inequality
\[
\mathrm{area} \bigl( \set{\hat{u}(k)}{u\in \phi_t(B_r(u_0))} \bigr) \geq \pi r^2
\]
holds,
where ``area'' stands for the standard area on $\C$ (here the question is non trivial only when $u_0\neq 0$, because the flow of the nonlinear Schr\"odinger equation (\ref{nls}) preserves the $L^2$ norm, so the equality holds for every $t\in \R$ when $u_0=0$). The above inequality says that during the evolution we cannot obtain a better determination of the value of a single Fourier coefficient than the one we have for $t=0$, even if we are willing to loose control on the value of all the other Fourier coefficients. As Kuksin observed in \cite{kuk95}, the validity of such a non-squeezing inequality forbids the existence of steady states which are attractors for an open set of initial conditions, and  forbids also a certain kind of {\em energy transfer}, that is, a certain way in which the energy can be spread from low to high Fourier modes. We refer to \cite{kuk95} for precise explanations. See also \cite{kuk95b}, \cite{bou95} and \cite{ckstt10} for results showing that other forms of energy transfer - in Sobolev spaces of higher regularity, where the symplectic form is weak - are instead to be expected in the case of nonlinear PDEs.

Having discussed the meaning of the non-squeezing phenomenon on symplectic Hilbert spaces, we now review the known results about its validity.
In \cite{kuk95} Kuksin has proved that the non-squeezing theorem holds for Hamiltonian PDEs whose flow is a smooth compact perturbation of a linear flow of the form $e^{JAt}$, where $A$ is an unbounded operator which is self-adjoint on the complex Hilbert space $(\mathbb{H},J,(\cdot,\cdot))$ and semi-simple. This class of PDEs includes, for instance, the nonlinear wave equation on $\T$ with a smooth nonlinearity having polynomial growth, the nonlinear wave equation on $\T^2$  with nonlinearity of degree at most four (see \cite{bou95}), the membrane equation on $\T^2$ with a smooth nonlinearity having polynomial growth, and the Schr\"odinger equation on $\T^n$ with a nonlinearity of convolution type. Shortly afterwards, Bourgain \cite{bou94} has proved the non-squeezing theorem for the cubic Schr\"odinger equation on $\T$, whose flow cannot be seen as a compact perturbation of a linear one. More recently, the non-squeezing theorem has been confirmed for the KdV equation by Colliander, Keel, Staffilani, Takaoka and Tao \cite{ckstt05b} and for the BBM equation by Roum\'egoux \cite{rou10}. In all these papers the conclusion is deduced from Gromov's theorem by finding suitable finite dimensional approximations of the infinite dimensional flow. In \cite{bou94},  \cite{ckstt05b} and \cite{rou10} these finite dimensional approximations are quite delicate and rely on special algebraic properties of the equation under consideration, which imply suitable cancellations that reduce the interactions between low and high frequencies. Conservation laws also play a fundamental role.

\medskip

In this paper, we deal with general symplectomorphisms on a symplectic Hilbert space. After recalling the basic notions of symplectic geometry on infinite-dimensional Hilbert spaces in Sections \ref{sec1} and \ref{symchar}, we show some of the new phenomena which can arise in the infinite-dimensional setting: In Section \ref{example} we construct a symplectomorphism which maps a bounded closed convex neighborhood of the origin into its interior part.  This example is constructed by starting from a convex coercive Hamiltonian which does not admit non-constant periodic orbits.

The main result of this paper, whose proof is contained in Sections 4 to 8, is an infinite dimensional non-squeezing theorem which is not based on finite dimensional approximations and whose proof does not reduce to Gromov's theorem:

\begin{mainthm}
\label{uno}
Let $\varphi:B_r \rightarrow \varphi(B_r)\subset \mathbb{H}$ be a smooth symplectomorphism such that $\varphi(B_r)$ is convex. Assume moreover that the differentials up to the third order of $\varphi$ and $\varphi^{-1}$ are bounded. Let $P$ be the symplectic projector onto a 2-dimensional linear subspace $\mathbb{H}_0\subset \mathbb{H}$. Then 
\[
\mathrm{area}_{\omega}(P\varphi(B_r)) \geq \pi r^2,
\]
where the area form on $\mathbb{H}_0$ is induced by the restriction of $\omega$.
\end{mainthm}

The crucial assumption here is that $\varphi(B_r)$ should be convex. Of course this assumption prevents the application of this result to the long time evolution of a ball by a nonlinear flow. However, it provides an obstruction to what general infinite dimensional Hamiltonian flows can do to balls of a fixed size on short time scales, or equivalently to sufficiently small balls on a large but fixed time scale. By the convexity of the ball, the boundedness assumption on the differentials of $\varphi$ and $\varphi^{-1}$ is equivalent to the boundedness of the maps $d^3 \varphi$ and $(d\varphi^{-1})$ on $B_r$. This assumption is not very restrictive: The differentials of any order of the flow of a typical Hamiltonian PDE which is well-posed in its Darboux phase space are bounded on bounded sets. Here is a corollary of Theorem \ref{uno}:

\begin{maincor}
Consider the one-dimensional periodic nonlinear Schr\"odinger equation
\begin{equation}
\label{nls4}
- i \partial_t u + \partial_{xx} u = \partial_{\bar{u}} F(t,x,u,\bar{u}), \qquad x\in \T,
\end{equation}
where $F$ is a polynomial in the last two variables of degree at most 4 with coefficients depending smoothly on $(t,x)\in \R\times \T$ and is real  (i.e.\ $\bar{F}(t,x,u,v)=F(t,x,\bar{v},\bar{u})$). Assume that the solutions of the initial value problem for (\ref{nls4}) are well defined for every  $u(0,\cdot)$ in a $L^2$-ball of radius $r_0$ centered in $u_0\in L^2(\T,\C)$ and for every $0\leq t \leq T$. Then there exists $r_1=r_1(u_0,T)\leq r_0$ such that for every $r\in (0,r_1]$ and every $t\in [0,T]$ the flow map $\phi_t$ of (\ref{nls4}) satisfies the non-squeezing property
\[
\mathrm{area} \bigl( \set{\hat{u}(k)}{u\in \phi_t(B_r(u_0))} \bigr) \geq \pi r^2,
\]
for every $k\in \Z$.
\end{maincor} 

The non-trivial point in the above corollary is the fact that the non-squeezing inequality holds with the optimal constant $\pi$: the analogous inequality with any smaller constant would hold for sufficiently small $r$ just by the continuity of the flow.
The above corollary follows immediately from Theorem \ref{uno} and from the local well-posedness of (\ref{nls4}) in the Darboux phase space $L^2(\T,\C)$, which under the above assumptions on $F$ has been proved by Bourgain in \cite{bou93} (see also \cite[Section 3]{bou95} for the generality considered here). Indeed, if $r_1$ is small enough, the image of any ball of radius $r\leq r_1$ around $u_0$ remains convex up to time $T$, being $C^2$-close to the evolution of the same ball with respect to the differential of $\phi_t$ at $u_0$. The assumption on the existence up to time $T$ is necessary because solutions of this general nonlinear Schr\"odinger equation might blow up. When the polynomial $F$ is a function of $|u|$, then the flow preserves the $L^2$-norm, and this assumption is automatically fulfilled. The flow of the equation (\ref{nls4}) cannot be seen as a compact perturbation of a linear flow, so Kuksin's result cannot be applied here. Bourgain \cite{bou95} has shown that this flow admits finite dimensional approximations which, although not uniform enough to prove the non-squeezing property, permit to deduce weaker statements, such as the fact that the diameter of the evolution of a ball cannot shrink to zero. To the best of our knowledge, it is not known whether (\ref{nls4}) satisfies the non-squeezing property for balls of arbitrary size, except for the case 
\[
F(t,x,u,\bar{u}) = a(t,x) |u|^2 + b(t,x) |u|^4,
\]
where $a$ and $b$ are smooth functions, which is considered by Bourgain in the already mentioned \cite{bou94}. Bourgain's proof builds on the fact that the preservation of the $L^2$-norm and suitable cancellations imply better approximation properties of the finite dimensional reductions. 

\medskip

In order to prove Theorem \ref{uno}, we construct a {\em symplectic capacity} for bounded closed convex neighborhoods of the origin in $\mathbb{H}$. The notion of symplectic capacity for subsets of $\R^{2n}$ was introduced by Ekeland and Hofer in \cite{eh89} and further developed by many authors (see the book of Hofer and Zehnder  \cite{hz94} for a comprehensive introduction).  A symplectic capacity on the class of compact convex subsets $C$ of $\R^{2n}$ with smooth boundary can be simply defined as the minimal action
\[
\mathbb{A}(x) := \int_{\T} x^* \lambda
\]
over all the closed characteristics on the boundary of $C$. 
Here $\lambda$ denotes a primitive of $\omega$ and a smooth curve $x: \T \rightarrow \partial C$ is said to be a closed characteristic if for every $t\in \T$ the vector $\dot{x}(t)$ is a positive multiple of $J n_C(x(t))$, where $n_C(x)$ denotes the outer normal to $C$ at $x$. Equivalently, $x$ is the time-reparametrization of a periodic Hamiltonian orbit of a smooth Hamiltonian having $\partial C$ as regular energy level and which increases in the outer normal direction. The existence of a closed characteristic on $\partial C$ was first proved by Weinstein in \cite{wei78}. The action of every closed characteristic is positive, and a closed characteristic with minimal action always exists. The fact that the minimal action of a closed characteristic on $\partial C$ coincides with the Ekeland-Hofer capacity of $C$, as defined in \cite{eh89}, is observed explicitly in \cite[Proposition 3.10]{vit89}. 

Let $C$ be a bounded closed convex neighborhood of the origin with smooth boundary in the symplectic Hilbert space $(\mathbb{H},\omega)$. When $\mathbb{H}$ is infinite-dimensional, $\partial C$ may have no closed characteristics at all, as an example in Section \ref{example} shows. However, we can define the symplectic capacity of $C$ as the positive number
\begin{equation}
\label{lacap}
c_{\mathbb{H}} (C) := \Bigl( 4 \sup \set{\mathbb{A}^*(\xi)}{\xi: \mathbb{T} \rightarrow \mathbb{H}^* \mbox{ absolutely continuous, } \dot{\xi} \in C^0 \mbox{ a.e.}} \Bigr)^{-1},
\end{equation}
where $\mathbb{A}^*$ denotes the action of a closed curve in the dual of $\mathbb{H}$, and $C^0\subset \mathbb{H}^*$ denotes the polar set of $C$. Notice that $c_{\mathbb{H}} (C)$ has the dimension of an area: Indeed, the symplectic action of a closed curve in $\mathbb{H}^*$ has the dimensions of the inverse of an area.

When $\mathbb{H}$ is finite-dimensional, the right hand-side of (\ref{lacap}) is a variational characterization of the Ekeland-Hofer capacity of $C$. This variational characterization does not seem to be explicitly present in the literature, but it is in the spirit of Clarke's and Ekeland's use of Fenchel duality to detect closed characteristics on $\partial C$ (see \cite{cla79}, \cite{ce80}, \cite{cla81} and \cite{eke90}). It is an easy matter to show that (\ref{lacap}) is equivalent to the more familiar formula 
\begin{equation}
\label{lacap2}
c_{\mathbb{H}} (C) = \inf \set{ \frac{1}{4} \int_{\T} \mu_{C^0}^2(\dot{\xi})\, dt}{\xi: \T \rightarrow \mathbb{H}^* \mbox{ abslutely continuous, } \mathbb{A}^*(\xi)=1},
\end{equation}
where $\mu_{C^0}: \mathbb{H}^* \rightarrow \R$ is the the Minkowski gauge of $C^0$ (see Section \ref{eqsec} below). When $\mathbb{H}$ is finite-dimensional the supremum in (\ref{lacap}) (resp.\ the infimum in (\ref{lacap2})) is achieved by a curve $\xi$ such that $-\Omega^{-1} \xi$ is homothetic to a closed characteristic with minimal action on $\partial C$ (see Theorem \ref{minicara} below). In our infinite-dimensional setting, this supremum (resp.\ infimum) is in general not achieved, but defines nevertheless a symplectic capacity on the set of closed bounded convex neighborhoods of the origin in $\mathbb{H}$, i.e.\ a function which satisfies the following properties:
\begin{enumerate}[(i)]
\item {\em (Monotonicity)} If $C_1\subset C_2$ then $c_{\mathbb{H}}(C_1) \leq c_{\mathbb{H}}(C_2)$.
\item {\em (Homogeneity)} $c_{\mathbb{H}} (rC) = r^2 c_{\mathbb{H}}(C)$ for every $r>0$. 
\item {\em (Normalization)} If $B$ is the closed unit ball of $\mathbb{H}$, then $c_{\mathbb{H}} (B) = \pi$.
\item {\em (Projection)} Let $P$ be the symplectic projector onto a symplectic closed linear subspace $\mathbb{H}_0 \subset \mathbb{H}$. Then $c_{\mathbb{H}_0}(PC) \geq c_{\mathbb{H}}(C)$.
\item {\em (Continuity)} The function $c_{\mathbb{H}}$ is continuous with respect to the Hausdorff metric.
\item {\em (Invariance)} Assume that $C$ has a regular boundary and is strongly convex.
Let $\varphi:C \rightarrow \varphi(C)\subset \mathbb{H}$ be a smooth symplectomorphism onto a convex neighborhood of the origin such that the differentials up to the third order of $\varphi$ and $\varphi^{-1}$ are bounded. Then $c_{\mathbb{H}}(\varphi(C)) =  c_{\mathbb{H}}(C)$.
\end{enumerate}

The non-squeezing Theorem \ref{uno} is an immediate consequence of the above properties and of the fact that, when $\dim \mathbb{H}_0=2$, the capacity $c_{\mathbb{H}_0}$ is just the area. 

Properties (i) to (v) follow quite easily from the definition (\ref{lacap}). The nontrivial part of our proof is to show that the invariance property (vi) holds. In the finite-dimensional case (vi) follows from the fact that symplectomorphisms preserve the closed characteristics and their action. When $\mathbb{H}$ is infinite-dimensional, $c_{\mathbb{H}}(C)$ cannot be interpreted as the action of a closed characteristic and the invariance property (vi) is nontrivial also for simple symplectomorphisms such as translations. The reason is that the polar of $\varphi(C)$ has little to do with the polar of $C$, unless $\varphi$ is linear. 

In Section \ref{secinv} we prove the invariance of $c_{\mathbb{H}}$ with respect to symplectomorphisms $\varphi$ which are positively 1-homogeneous maps, under the assumption that both $C$ and $\varphi(C)$ are regular and strongly convex. Here the main point is to show that minimizing sequences of (\ref{lacap2}) which are also Palais-Smale sequences are in a certain sense homothetic to ``almost closed characteristics of $\partial C$'', although in general they fail to converge to a curve which is homothetic to a true closed characteristic. 

In Section \ref{gensymp} we prepare the ground for the general case, by proving the following result, which might be of independent interest: If two smooth bounded convex neighborhoods of the origin $C_1$ and $C_2$ are symplectomorphic,  then $C_2$ is the image of $C_1$ by a positively 1-homogeneous symplectomorphism. The proof of the latter fact is based on a characterization of positively 1-homogeneous symplectomorphisms and on Moser's argument from \cite{mos65}. Property (vi) is then proved in Section \ref{geninv}, where a perturbation argument and property (v) allow us to remove the strong convexity assumption on $\varphi(C)$.

The existence of the symplectic capacity $c_{\mathbb{H}}$ allows us to prove also the following middle-dimensional non-squeezing result.

\begin{mainthm}
\label{due}
There exists a constant $\gamma>0$ with the following property.
Let $\varphi:B_r \rightarrow \varphi(B_r)\subset \mathbb{H}$ be a smooth symplectomorphism which satisfies the assumptions of Theorem \ref{uno}. Let $P$ be the symplectic projector onto a 2k-dimensional symplectic linear subspace $\mathbb{H}_0\subset \mathbb{H}$. Then 
\[
\mathrm{vol}_{\omega^k}(P\varphi(B_r)) \geq \gamma^{-1} \, \pi^k r^{2k},
\]
where the volume form on $\mathbb{H}_0$ is induced by the restriction of $\omega^k= \omega \wedge \dots \wedge \omega$.
\end{mainthm}

Indeed, once one has a symplectic capacity with the above properties, the above theorem follows from the fact that the Ekeland-Hofer capacity $c_{\R^{2k}}(C)$ of a convex set $C$ in $\R^{2k}$ can be bounded by its volume through the inequality
\begin{equation}
\label{vit}
c_{\R^{2k}}(C)^k \leq \gamma \, \mathrm{vol}_{\omega^k}(C),
\end{equation}
where $\gamma$ is an absolute constant which does not depend on $k$. The proof of latter inequality is due to Artstein-Avidan, Milman and Ostrover \cite{amo08} and holds for every symplectic capacity on convex subsets of $\R^{2n}$. Conjecturally, the value of the constant $\gamma$ in (\ref{vit}) - and hence in Theorem \ref{due} - is 1. This conjecture is due to Viterbo \cite{vit00}, who first proved the bound (\ref{vit}) with a constant $\gamma$ depending on the dimension $k$.

Unlike Theorem \ref{uno}, which in principle could hold also when $\varphi(B_r)$ is not convex, Theorem \ref{due} is not true, even in finite dimension, if we remove the convexity assumption and we project onto a proper symplectic subspace of dimension al least 4: in \cite{am13} the first author and Matveyev have constructed symplectomorphisms of $\R^6$ such that the projection of the image of the unit ball onto a symplectic 4-dimensional space has arbitrary small volume. By taking products with the identity mapping on suitable subspaces one gets counterexamples for any $1<k<\dim \mathbb{H}\leq +\infty$.

When the symplectomorphisms $\varphi$ is linear, the estimate of Theorem \ref{due} holds with the optimal constant $\gamma=1$ (see again \cite{am13} for the finite-dimensional case). We discuss this linear result in the appendix which concludes this paper, where we also prove a non-linear consequence, which roughly speaking says that one-parameter families of symplectomorphisms on a Hilbert space cannot have an invariant compact set which is a ``uniform attractor''. 

\medskip

We conclude this introduction with a general comment on the argument of the proof of our main result. All known proofs of the finite-dimensional non-squeezing theorem are based on some existence principle, either of $J$-holomorphic curves (as in Gromov's original proof \cite{gro85}) or of periodic orbits (as in Ekeland and Hofer's \cite{eh89} and in Viterbo's \cite{vit89} proofs). Our proof here is no exception: The proof of the crucial invariance property (vi) uses the existence of suitable Palais-Smale sequences, which are our replacement for periodic orbits. The fact that we are assuming the image of the ball to be convex allows us to use a dual variational principle, for which the existence of these Palais-Smale sequences is a simple consequence of the fact that the functional is bounded from below. In principle, the same idea could work and prove the non-squeezing theorem in a more general, if not in {\em the} general, case, by considering the direct action functional, but the difficulty there seems to be how to produce suitable Palais-Smale sequences.

\section{Symplectic structures on Hilbert spaces}
\label{sec1}

In this section we recall some basic facts about symplectic structures on infinite dimensional Hilbert spaces. A classical reference for these facts, in a more general Banach setting, is the book of Chernoff and Marsden \cite{cm74}. Let $\mathbb{H}$ be a real Hilbert space. A {\em symplectic structure} on $\mathbb{H}$ is given by a skew-symmetric continuous bilinear form 
\[
\omega:\mathbb{H} \times \mathbb{H} \rightarrow \R 
\]
which is {\em non-degenerate}, meaning that the associated bounded linear operator $\Omega:\mathbb{H} \rightarrow \mathbb{H}^*$,  which is defined as
\begin{equation}
\label{defJ}
\langle \Omega x,y \rangle = \omega(x,y) \qquad \forall x,y\in \mathbb{H},
\end{equation}
where $\langle \cdot,\cdot \rangle$ denotes the duality pairing, is an isomorphism. By the skew symmetry of $\omega$ we have $\Omega^*=-\Omega$, where $\Omega^* : \mathbb{H}^{**} \rightarrow \mathbb{H}^*$ is the adjoint of $\Omega$ and we are identifying the reflexive space $\mathbb{H}$ with its bidual.

The choice of a Hilbert inner product $(\cdot,\cdot)$ on $\mathbb{H}$ determines a bounded linear operator $J : \mathbb{H} \rightarrow \mathbb{H}$ such that
\[
(Jx,y) = \omega(x,y) \qquad \forall x,y\in \mathbb{H}.
\]
Being the composition of $\Omega$ by the isomorphism $\mathbb{H}^* \cong \mathbb{H}$ induced by the inner product, $J$ is also an isomorphism. The skew symmetry of $\omega$ now reads $J^T = - J$, where $J^T: \mathbb{H} \rightarrow \mathbb{H}$ is the transposed operator with respect to the inner product. When one of the following equivalent conditions hold, the inner product $(\cdot,\cdot)$ is said to be compatible with the symplectic structure $\omega$:
\begin{enumerate}[(i)]
\item $\Omega$ is an isometry (where $\mathbb{H}^*$ is endowed with the dual norm).
\item $J$ is an isometry.
\item $J$ is a complex structure (that is, $J^2=-I$). 
\end{enumerate}
The equivalence of (i) and (ii) is clear, because the isomorphism $\mathbb{H}^*\cong \mathbb{H}$ which is induced by the inner product is an isometry. The equivalence between (ii) and (iii) follows from the fact that $J^T=-J$. 

When $\mathbb{H}$ is finite-dimensional, the existence of a compatible inner product can be deduced from the existence of a symplectic basis, which can be constructed by induction on the dimension. In general, a compatible inner product can be easily constructed by using operator calculus. Indeed, starting from any inner product $(\cdot,\cdot)_0$ on $\mathbb{H}$, the symplectic form $\omega$ is represented as $\omega(x,y)=(J_0 x,y)_0$, where $J_0=-J_0^T$ is an invertible bounded operator. Let $A$ be the (symmetric, positive) square root of the positive operator $-J_0^2=J_0^TJ_0$. Since $J_0$ is invertible, $A$ is positive definite, and since $J_0$ is normal, it commutes with $A$, so that $J:=A^{-1}J_0$ is a square root of $-I$. The equivalent inner product $(x,y):=(Ax,y)_0$ is then compatible with $\omega$, as 
\[
\omega(x,y)=(J_0x,y)_0=(AA^{-1}J_0x,y)_0=(Jx,y).
\]

A linear isomorphism $\Phi$ from  $(\mathbb{H}_1,\omega_1)$ onto $(\mathbb{H}_2,\omega_2)$ is said to be symplectic if
\[
\omega_2(\Phi x,\Phi y) = \omega_1(x,y) \qquad \forall x,y\in \mathbb{H}_1.
\]
If $\Omega_1: \mathbb{H}_1 \rightarrow \mathbb{H}_1^*$ and $\Omega_2: \mathbb{H}_2 \rightarrow \mathbb{H}_2^*$ are the isomorphisms which are associated to $\omega_1$ and $\omega_2$, the above condition can be rewritten as
\[
\Phi^* \Omega_2 \Phi = \Omega_1.
\]
Using compatible inner products on $\mathbb{H}_1$ and $\mathbb{H}_2$ and denoting by $J_1$ and $J_2$ the corresponding complex structures, this is equivalent to
\[
\Phi^T J_2 \Phi = J_1.
\]
The basic example of symplectic Hilbert spaces is given by complex Hilbert spaces, as in the following:

\begin{ex}
\label{complex}
Let $\mathbb{H}_{\C}$ be a {\rm complex} Hilbert space with (Hermitian) inner product $(\cdot,\cdot)_{\C}$ and Hilbert norm $\|x\|_{\C} :=\sqrt{(x,x)_{\C}}$. Then the real normed vector space $\mathbb{H}$ obtained from $\mathbb{H}_{\C}$ by restriction of the scalar field is a real Hilbert space with respect to the (real) inner product $(x,y):=\re (x,y)_{\C}$. The linear operator  $J$ on $\mathbb{H}$ mapping $x$ into $ix$ is an orthogonal square root of $-I$, and 
\[
\omega(x,y):=-\im(x,y)_{\C}=(Jx,y)
\]
defines a symplectic structure on $\mathbb{H}$, compatible with the scalar product $(\cdot,\cdot)$.
\end{ex}

Conversely, every real symplectic Hilbert(able) space $(\mathbb{H},\omega)$ can be seen as arising from the above construction. Indeed, one considers a compatible inner product $(\cdot,\cdot)$ on $\mathbb{H}$ and the associated skew-symmetric isometry $J$ which represents $\omega$ with respect to it. Then $J$ is an orthogonal square root of $-I$, and $\mathbb{H}$ as a $\R[J]$-module is a complex vector space $\mathbb{H}_{\C}$. The formula   
\[
(x,y)_{\C} :=(x,y)-i(Jx,y)
\] 
defines a Hermitian inner product on $\mathbb{H}_{\C}$ that induces the same norm as $(\cdot,\cdot)$, and $(\mathbb{H},(\cdot,\cdot), J)$ is obtained from $\mathbb{H}_{\C}$ by restriction of scalars as described in Example \ref{complex}.

Compatible inner products on a real symplectic Hilbert space are of course not unique. However the corresponding unit balls are all linearly symplectmorphic: If $(\cdot,\cdot)_1$ and $(\cdot,\cdot)_2$ are compatible inner products on $(\mathbb{H},\omega)$ with associated complex structures $J_1$ and $J_2$, then any complex linear isometry from $(\mathbb{H},J_1,(\cdot,\cdot)_1 - i (J\cdot,)_1)$ onto $(\mathbb{H},J_2,(\cdot,\cdot)_2 - i (J\cdot,)_2)$ is a symplectic isomorphism mapping the unit ball of $(\cdot,\cdot)_1$ onto the unit ball of $(\cdot,\cdot)_2$. 

Two symplectic Hilbert spaces $(\mathbb{H}_1,\omega_1)$ and $(\mathbb{H}_2,\omega_2)$ with the same Hilbert dimension are symplectically isomorphic: Indeed, there are complex Hilbert space structures $((\cdot, \cdot)_1, J_1)$ on $\mathbb{H}_1$ and $((\cdot, \cdot)_2,J_2)$ on $\mathbb{H}_2$ such that $\omega_1(x,y)=-\im(x,y)_1$ and $\omega_2(x,y)=-\im(x,y)_2$. These two complex Hilbert spaces structures are certainly isomorphic, and every complex linear (surjective) isometry between them is also a symplectic isomorphism.

\medskip

From now on $(\cdot,\cdot)$ denotes a (real) Hilbert product on $\mathbb{H}$ compatible with $\omega$, $\|\cdot\|$ is the associated norm, and $J$ is the bounded operator on $\mathbb{H}$ representing $\omega$ with respect to $(\cdot,\cdot)$. The dual norm on $\mathbb{H}^*$ is denoted by $\|\cdot\|_*$.

A closed linear subspace $\mathbb{H}_0$ of $\mathbb{H}$ is said to be {\em symplectic} if $\omega$ restricts to a symplectic form on $\mathbb{H}_0$. In this case, the {\em symplectic projector} onto $\mathbb{H}_0$ is the projector onto $\mathbb{H}_0$ along its symplectic orthogonal complement 
\[
\mathbb{H}_0^{\perp_{\omega}} := \set{x\in \mathbb{H}}{\omega(x,y) = 0 \; \forall y\in \mathbb{H}}.
\]
It is straightforward to check that $\mathbb{H}_0$ is symplectic if and only if
\[
\mathbb{H} = J  \mathbb{H}_0 \oplus \mathbb{H}_0^{\perp},
\]
where $\mathbb{H}_0^{\perp}$ denotes the orthogonal complement of $\mathbb{H}_0$ with respect to the inner product. Using the identity $J^2=-I$, the above condition is equivalent to
\[
\mathbb{H} = \mathbb{H}_0 \oplus J \mathbb{H}_0^{\perp}.
\]
This proves the following:

\begin{lem}
\label{fincodim}
Let $\omega$ be a symplectic form on the Hilbert space $\mathbb{H}$, and let $\mathbb{H}_0$ be a closed linear subspace of $\mathbb{H}$. Then $\mathbb{H}_0$ is symplectic if and only if $\mathbb{H}_0^{\perp}$ is symplectic.
\end{lem}

We conclude this section by extending to this infinite-dimensional setting other  familiar notions from finite-dimensional symplectic geometry. The {\em Liouville vector field} $Y: \mathbb{H} \rightarrow \mathbb{H}$ is the radial vector field defined by
\[
Y(x) := \frac{x}{2}.
\]
It has the property that the contraction of $\omega$ along $Y$ is a primitive of $\omega$: If $\lambda$ is the smooth one-form on $\mathbb{H}$ defined by $\lambda := \imath_Y \omega$, that is
\[
\lambda(x)[u] = \omega(Y(x),u) = \frac{1}{2} \omega(x,u)=\frac{1}{2} (Jx,u)\qquad  \forall x,u\in \mathbb{H},
\]
then $d\lambda=\omega$. Indeed:
\[
d\lambda(x)[u,v] = d (\lambda(\cdot)[v])(x) [u] - d (\lambda(\cdot)[u])(x) [v] = \frac{1}{2} \omega(u,v) - \frac{1}{2} \omega(v,u) = \omega(u,v).
\]
The one-form $\lambda$ is the standard {\em Liouville form} of $(\mathbb{H},\omega)$. Its kernel at $x\neq 0$ is the hyperplane
\begin{equation}
\label{kerlambda}
\ker \lambda(x) = J \bigl( (\R x)^{\perp} \bigr).
\end{equation}
Let $\T:=\R/\Z$.  We denote by 
\[
\mathbb{A}(x) := \int_{\T} x^*(\lambda) = - \frac{1}{2} \int_{\T} (J\dot{x}(t),x(t))\, dt = - \frac{1}{2} \int_{\T} \langle \Omega \dot{x}(t),x(t)\rangle \, dt
\]
the symplectic action of the absolutely continuous closed curve $x:\T \rightarrow \mathbb{H}$. By Stokes theorem, $\mathbb{A}(x)$ coincides with the integral of $\omega$ over any oriented disc in $\mathbb{H}$ which is bounded by the closed oriented curve $x$.

\section{Symplectomorphisms, Hamiltonian vector fields and characteristics}
\label{symchar}

Let $1\leq k \leq \infty$. A $C^k$ diffeomorphism $\varphi: A \rightarrow A'$ between open subsets $A,A'$ of $\mathbb{H}$ is called a $C^k$ {\em symplectomorphism} if $\varphi^* \omega = \omega$, that is, if  its differential $d\varphi(x)$ is a symplectomorphism for every $x\in A$.
A map $\varphi:X\rightarrow X'$ between arbitrary subsets $X,X'$ of $\mathbb{H}$ is said to be a $C^k$ symplectomorphism if it is a homeomorphism and there are open neighborhoods $A$ and $A'$ of $X$ and $X'$ such that $\varphi$ extends to a $C^k$ symplectomorphism from $A$ onto $A'$.

Symplectomorphisms between simply connected domains preserve the symplectic action of closed curves: If $x: \T \rightarrow A \subset \mathbb{H}$ is an absolutely continuous closed curve and $\varphi: A \rightarrow A'$ is a $C^1$ symplectomorphism between simply connected domains, then the closed 1-form $\varphi^* \lambda - \lambda$ is the differential of some function $h$, and hence
\[
\mathbb{A}(\varphi(x)) = \int_{\T} x^*(\varphi^* \lambda) = \int_{\T} x^*(\lambda + dh) = \int_{\T} x^*(\lambda) = \mathbb{A}(x).
\]

The standard way of producing symplectomorphisms on $(\mathbb{H},\omega)$ is the integration of Hamiltonian vector fields: A differentiable real function $H$ on an open subset $A$ of $\mathbb{H}$ defines the {\em   Hamiltonian vector field} $X_H : A \rightarrow \mathbb{H}$ by the identity
\[
\imath_{X_H} \omega = -dH.
\]
Equivalently, $X_H$ can be expressed in terms of the compatible inner product as
\[
X_H = J \nabla H.
\]
The Hamiltonian ODE
\[
\dot{x}(t) = X_H(x(t))
\]
can be equivalently written as
\[
- J \dot{x}(t) = \nabla H(x(t)) \qquad \mbox{or} \qquad 
- \Omega \dot{x}(t) = dH(x(t)).
\]
The local flow $\phi_t^{X_H}$ of a Hamiltonian vector field $X_H$ of class $C^k$, $k\geq 1$, preserves the energy levels $H^{-1}(c)$ and consists of $C^k$ symplectomorphisms. The latter fact remains true also when integrating time-dependent Hamiltonian vector fields.

Let $S$ be a hypersurface in $\mathbb{H}$ of class $C^1$. The kernel of the restriction of $\omega$ to $S$ defines a 1-dimensional distribution on $S$,
\[
\mathscr{D}_S(x) = \ker \omega|_{T_x S} = J (T_x S)^{\perp},
\]
which is called the {\em characteristic distribution of $S$}. When $S$ is the boundary of a set $C$, this distribution is {\em oriented} by declaring $J n_C(x)\in \mathscr{D}_S(x)$ to be a positive vector, where $n_C(x)$ is the outer unit normal. In this case, a $C^1$ curve $x:\R \rightarrow \partial S$ which is everywhere tangent to $\mathscr{D}_S$ and positively oriented is said to be a {\em characterstic curve} of $S$. If $S$ is a level set of a function $H\in C^1(\mathbb{H})$, then the Hamiltonian vector field $X_H$ is tangent to $\mathscr{D}_S$ on $S$, and it is positively oriented when $S$ is a regular level set of $H$ and $H<H|_{S}$ on the open set which is bounded by $S$.  Therefore, characteristic curves are, up to an orientation preserving time reparametrization, solutions of the Hamiltonian ODE defined by $X_H$ having some (and therefore every) point on $S$. 
In particular, {\em closed characteristics} correspond to periodic orbits of $X_H$ on $S$.

Now consider the case $S=\partial C$, where $C$ is the closure of an open convex set with a $C^1$ boundary.  Let us show that the action of a closed characteristic $x: \T \rightarrow \partial C$ is always a positive number. Up a translation, we may assume that $0$ belongs to the interior of $C$. Moreover, there holds
\[
\dot{x}(t) = f(t) J n_{C}(x(t))
\]
for some positive function $f$, and we obtain
\[
\mathbb{A}(x) = - \frac{1}{2} \int_{\T} (J \dot{x},x) \, dt =  \frac{1}{2} \int_{\T} f(x)(n_{C}(x),x) \, dt >0,
\]
because the convexity of $C$ implies that $(n_{C}(x),x)$ is a positive function. As a matter of fact, if $r>0$ is such that $rB\subset C$, where $B$ denotes the closed unit ball of $\mathbb{H}$,  then
\begin{equation}
\label{normale}
(n_C(x),x) \geq r \qquad \forall x\in \partial C.
\end{equation}
Indeed, from the inclusion
\[
rB \subset C \subset \set{x+u\in \mathbb{H}}{(u,n_C(x))\leq 0}
\]
we deduce that $rn_C(x)$ has the form $x+u$ with $(u,n_C(x))\leq 0$, and the inequality
\[
r = (rn_C(x),n_C(x)) = (x,n_C(x)) + (u,n_C(x)) \leq (x,n_C(x))
\]
implies (\ref{normale}).

\section{An example}
\label{example}

In this section we show that symplectomorphisms on infinite-dimensional Hilbert spaces can exhibit behaviours which are forbidden in finite dimensions. 

Let $0<a<b$ and set $\mathbb{H}=L^2((a,b),\C)$, which we see as a real Hilbert space with the inner product
\[
(u,v) := \re \int_a^b u\bar{v}\, dx \qquad \forall u,v\in \mathbb{H}.
\]
As we have seen in Example \ref{complex}, the continuous skew symmetric bilinear form
\[
\omega(u,v) := - \im \int_a^b u\bar{v}\, dx \qquad \forall u,v\in \mathbb{H},
\]
is a symplectic form on $\mathbb{H}$, and $(\cdot,\cdot)$ is a compatible inner product whose associated complex structure $J$ is the multiplication by $i$.

On $\mathbb{H}$ we consider the smooth Hamiltonian
\[
H : \mathbb{H} \rightarrow \R, \qquad H(u) := \frac{1}{2}  \int_a^b x |u(x)|^2\, dx.
\]
Since $a>0$, this function is a positive definite quadratitc form, and its sublevels $\{H\leq c\}$, $c>0$, are bounded convex neighborhoods of the origin having the ellipsoid $H^{-1}(c)$ as boundary. 

The Hamiltonian vector field of $H$ is
\[
X_H(u) = i M_x u,
\]
where $M_x: \mathbb{H} \rightarrow \mathbb{H}$ is the multiplication operator by the function $x$. Therefore, the Hamiltonian ODE which is induced by $H$ is
\[
\frac{\partial u}{\partial t} (t,x) = i x\, u(t,x),
\]
and its flow is given by
\[
\phi_t^{X_H}(u)(x) = e^{itx} u(x)\qquad \forall u\in \mathbb{H}.
\]
This flow consists of linear isometries. It is easy to see that it has no periodic orbits other than the trivial one $u=0$. Indeed, if  the orbit of $u$ is $T$-periodic, then 
\[
e^{iT x} u(x) = u(x) \quad \mbox{a.e.}.
\]
It follows that 
\[
T x \in 2\pi \Z \qquad \mbox{for a.e. } x\in \mathrm{supp}\, u,
\]
that is,
\[
\mathrm{supp}\, u \subset (2\pi/T) \,\Z,
\]
so the support of $u$ must have measure zero. Therefore, the sublevels $\{H\leq c\}$, $c>0$, are examples of bounded convex neighborhoods of the origin whose boundary has no closed characteristics (in contrast to what happens in finite dimension, where the boundary of a bounded convex domain always admits closed characteristics, as proved by Weinstein in \cite{wei78}). 
More generally, every ellipsoid which is defined by the quadratic form associated to a complex linear self-adjoint operator with no eigenvalues has no closed characteristics.

The above Hamiltonian flow has also other peculiar properties, which a Hamiltonian flow on a finite-dimensional space with compact energy levels could not have: The only recurrent point of $\phi^{X_H}$ is $u=0$, and the flow has a strict Lyapunov function on $\mathbb{H} \setminus \{0\}$. Indeed, the following result holds.

\begin{prop}
\label{lyapunov}
There exists a positively 2-homogeneous function $K: \mathbb{H} \rightarrow \R$ which is smooth on $\mathbb{H}\setminus \{0\}$, bounded on bounded sets, and such that
\[
dK(u)[X_H(u)] < 0 \qquad \forall u\in \mathbb{H}\setminus \{0\}.
\]
Furthermore, there exists a constant $c$ such that $\|dK(u)\|_* \leq c \|u\|$ for every $u\in \mathbb{H}$.
\end{prop}

Before proving this proposition, let us derive a consequence: From the identity
\[
dH(u)[X_{K}(u)] = -\omega(X_H(u),X_K(u)) = \omega(X_K(u),X_H(u)) = - dK(u)[X_H(u)]
\]
we deduce that
\[
dH(u)[X_{-K}(u)] = -  dH(u)[X_{K}(u)] = dK(u)[X_H(u)] < 0 \qquad \forall u\in \mathbb{H}\setminus \{0\}.
\]
Therefore, $H$ is a strict Lyapunov function for the Hamiltonian flow of $-K$ on $\mathbb{H}\setminus \{0\}$. Notice that this flow is globally defined, since 
\[
\|X_{-K}(u)\|=\|dK(u)\|_* \leq c \|u\| \qquad \forall u\in \mathbb{H}.
\]
The flow of $X_{-K}$ at time $t>0$ is a symplectomorphism which maps a bounded closed neighborhood of the origin into its interior part:
\[
\phi_t^{X_{-K}} ( \{ H\leq c\} ) \subset \{H<c\} \qquad \forall c>0, \; \forall t>0.
\]
Also such a behavior is of course impossible when $\mathbb{H}$ is finite-dimensional, due to the conservation of volume.

\begin{rem}
There is a certain freedom in the construction of the Lyapunov function $K$, but it seems that in every case the flow of $X_{-K}$ cannot squeeze the ellipsoids $\{H\leq c\}$ in a uniform way: The properties of the symplectic capacity $c_{\mathbb{H}}$ which are established in the following sections imply that as long as $\phi_t^{X_{-K}}(\{H\leq c\})$ remains convex, it cannot be contained in some ellipsoid $\{H\leq c'\}$ with $c'<c$, because the capacity of the latter set is strictly less than the capacity of $\{H\leq c\}$.
\end{rem}

The remaining part of this section is devoted to the proof of Proposition \ref{lyapunov}. This proof relies on the following:

\begin{lem}
\label{tubolar}
For every $u_0\in H^{-1}(1)$ there exists a positive number $r_0=r_0(u_0)$ such that the following is true: Denote by $D_{r_0}$ the closed ball of radius $r_0$ in the orthogonal complement to $X_H(u_0)$. Then the map
\[
\psi:\R \times D_{r_0} \rightarrow \mathbb{H}, \qquad (t,v) \mapsto \phi^{X_H}_t(u_0+v),
\]
is a smooth diffeomorphism onto a closed neighborhood $U_0$ of the orbit of $u_0$. Moreover, for every integer $k\geq 1$ the $k$-th differentials of $\psi$ and $\psi^{-1}$ are uniformly bounded by a constant which does not depend on $u_0\in H^{-1}(1)$.
\end{lem}

\begin{proof}
The proof makes use of the following observation: For every $u$ and $v$ in $\mathbb{H}$ there holds
\begin{equation}
\label{nonrec}
\lim_{t\rightarrow\pm \infty} \|\phi_t^{X_H}(u)-v\|^2 = \|u\|^2 + \|v\|^2.
\end{equation}
Indeed, this follows from the formula
\[
\|\phi_t^{X_H}(u)-v\|^2 = \|e^{itx} u - v\|^2 = \|u\|^2 + \|v\|^2 - 2 \re \int_a^b e^{itx} u(x) \bar{v}(x) \, dx,
\]
because the last integral is infinitesimal for $t\rightarrow \pm \infty$ by the Riemann-Lebesgue Lemma.

Denote by $V$ the orthogonal complement to $X_H(u_0)=ixu_0$. The differential of $\psi$ at $(t,v)$ is the linear mapping
\[
d\psi(t,v)[(s,w)] = e^{itx}(ix(u_0+v)s + w) \qquad \forall (s,w)\in \R\times V.
\]
The equation $d\psi(t,v)[(s,w)]=u$, $u\in \mathbb{H}$, can be uniquely solved for $(s,w)$ by
\[
\begin{split}
s &= \frac{(e^{-itx}u,ixu_0)}{\|xu_0\|^2 + (xv,xu_0)}, \\
w &=e^{-itx} u - ix (u_0 + v)s,
\end{split}
\]
provided that the denominator in the first formula does not vanish. By using the estimate
\[
\|xu_0\|^2 = \int_a^b x^2 |u_0|^2\, dx \geq a \int_a^b x|u_0|^2\, dx = 2a H(u_0) = 2a,
\]
we find that this denominator has the lower bound
\[
\|xu_0\|^2 + (xv,xu_0) \geq \|xu_0\| (\|xu_0\| - \|xv\|) \geq \sqrt{2a}( \sqrt{2a} - b \|v\|).
\]
We deduce that if $r_1 < \sqrt{2a}/b$ and $v\in D_{r_1}$, then $d\psi(t,v)$ is invertible, and its inverse is uniformly bounded. Therefore, $\psi$ is a local diffeomorphism on $\R \times D_{r_1}$, and for every integer $k\geq 1$ the $k$-th differentials of $\psi$ and of its local inverses are uniformly bounded.

Now we prove that $\psi$ is injective on $\R\times D_{r_0}$, if $r_0\leq r_1$ is small enough. Using the fact that $\psi|_{\{0\}\times V}$ is clearly injective and the fact that $\phi^{X_H}$ is a flow, it is enough to check that
\[
\psi(t,v) \neq \psi(0,w)
\]
whenever $t\neq 0$ and $v,w$ are vectors in $D_{r_0}$. Since $d\psi(0,0)$ is invertible, the inverse mapping theorem implies that there exists $r_2>0$ and $\tau>0$ such that
\begin{equation}
\label{vicini}
0< |t|\leq \tau, \quad v,w\in D_{r_2}  \quad \Rightarrow \quad
\psi(t,v) \neq \psi(0,w).
\end{equation}
By (\ref{nonrec}) we have
\[
\lim_{t\rightarrow \pm \infty} \|e^{itx} u_0 - u_0\|^2 = \lim_{t\rightarrow\pm \infty} \|\phi_t^{X_H}(u_0)-u_0\|^2 = 2\|u_0\|^2>0.
\]
Together with the fact that the curve $t\mapsto \phi^{X_H}_t(u_0)$ is injective, we deduce that the number
\[
\rho:= \inf_{|t|\geq \tau} \|e^{itx} u_0 - u_0\|
\]
is positive. Let $r_0$ be a positive number such that $r_0\leq \min\{r_1,r_2\}$ and $r_0<\rho/2$. We claim that if $v$ and $w$ belong to $D_{r_0}$ and $t\neq 0$, then $\psi(t,v)\neq \psi(0,w)$. If $|t|\leq \tau$ this follows from (\ref{vicini}). For $|t|\geq \tau$ we have
\[
\|\psi(t,v) - \psi(0,w)\| = \|e^{itx}(u_0+v) - (u_0+w)\| \geq \|e^{itx} u_0 - u_0\| - \|v\| - \|w\| \geq \rho - 2 r_0 >0.
\]
This proves our claim and the injectivity of $\psi$ on $\R \times D_{r_0}$.

In order to conclude that $\psi: \R \times D_{r_0} \rightarrow \mathbb{H}$ is a diffeomorphism onto a closed subset, there remains to show that $\psi$ is a proper map. Let $(t_n,v_n)$ be a sequence in $\R \times D_{r_0}$ such that $\psi(t_n,v_n)$ converges to some $u\in \mathbb{H}$. From (\ref{nonrec}) we have
\[
\lim_{t\rightarrow \pm \infty} \|\psi(t,v_n)-u\|^2 = \lim_{t\rightarrow \pm \infty} \|\phi_t^{X_H}(u_0+v_n)-u\|^2 = \|u_0+v_n\|^2 + \|u\|^2 >0,
\]
and hence the sequence $(t_n)$ is bounded. Up to a subsequence, we may assume that $(t_n)$ converges to some $t\in \R$, and we deduce that $(v_n)$ converges to $e^{-itx} u - u_0$. Therefore the map $\psi$ is proper.
\end{proof}

\begin{proof}[Proof of Proposition \ref{lyapunov}]  Since the space $H^{-1}(1)$ has the Lindel\"{o}f property, it has a countable subset $\{u_n\}_{n\in \N}$ such that the interiors of the closed neighborhoods $U_n$ of the orbits of $u_n$ which are constructed in lemma \ref{tubolar} cover $H^{-1}(1)$.  Denote by
\[
\psi_n : \R \times D_{r_n} \rightarrow U_n, \qquad (t,v) \mapsto \phi_t^{X_H}(u_n+v),
\]
the corresponding diffeomorphisms. By construction,
\[
d\psi_n(t,v)\left[ \frac{\partial}{\partial t} \right] = X_H(\psi_n(t,v)) \qquad \forall (t,v)\in \R \times D_{r_n}.
\]
Let $\chi: \R  \rightarrow \R$ be a smooth function such that $\chi=1$ on $(-\infty,0]$, $0<\chi<1$ on $(0,1)$, and $\chi=0$ on $[1,+\infty)$. Consider the functions
\[
\tilde{f}_n : \R \times D_{r_n}\rightarrow \R, \qquad \tilde{f}_n(t,v) := - \chi(\|v\|/r_n) \arctan t,
\]
and
\[
f_n: H^{-1}(1) \rightarrow \R, \quad f_n = \left\{ \begin{array}{ll} \tilde{f}_n \circ \psi_n^{-1} & \mbox{on } H^{-1}(1)\cap U_n, \\ 0 & \mbox{on } H^{-1}(1)\setminus U_n. \end{array} \right.
\]
The function $f_n$ is smooth, and if $u=\psi_n(t,v)$ belongs to the interior of $U_n$ then
\begin{equation}
\label{lyap}
\begin{split}
df_n(u)[X_H(u)] &= d\tilde{f}_n(\psi_n^{-1}(u))\circ d\psi_n^{-1}(u)[X_H(u)] = d\tilde{f}_n(t,v)\left[ \frac{\partial}{\partial t}\right]  \\ &= \frac{\partial \tilde{f}_n}{\partial t} (t,v) = - \frac{\chi(\|v\|/r_n)}{1+t^2} < 0.
\end{split}
\end{equation}
Moreover, for every integer $k\geq 0$ there is a constant $\tilde{c}_k$ such that
\[
\|d^k\tilde{f}_n(t,v)\| \leq \frac{\tilde{c}_k}{r_n^k} \qquad \forall (t,v)\in \R\times D_{r_n},
\]
and by the boundedness properties of the differentials of $\psi_n^{-1}$, 
\begin{equation}
\label{stima}
\|d^k f_n(u)\| \leq \frac{c_k}{r_n^k} \qquad \forall u\in H^{-1}(1),
\end{equation}
for some constant $c_k$.

The function $K:\mathbb{H} \rightarrow \R$ which is obtained as the positively 2-homogeneous extension of the function
\[
H^{-1}(0)\rightarrow \R, \qquad u\mapsto  \sum_{n\in \N} 2^{-n} e^{-1/r_n} f_n(u) 
\]
has the desired properties. Indeed, for every integer $k\geq 0$ the above series converges absolutely in the Banach space $C^k_b(H^{-1}(1))$ consisting of $C^k$ functions with bounded differentials up to order $k$, because
\[
\sum_{n\in \N} 2^{-n} e^{-1/r_n} \|d^kf_n\|_{\infty} \leq \sum_{n\in \N} 2^{-n} e^{-1/r_n}\frac{c_k}{r_n^k} \leq c_k \left(\max_{r> 0} \frac{e^{-1/r}}{r^k} \right) \sum_{n\in \N} 2^{-n} < +\infty,
\]
and hence it defines a bounded smooth function on $H^{-1}(1)$ with bounded differentials. The strict Lyapunov property of $K$ follows from (\ref{lyap}).
\end{proof}

\section{A symplectic capacity for convex subsets of $\mathbb{H}$}

The dual of $\mathbb{H}$ is also a symplectic vector space with the symplectic form
\[
\omega^* (\xi,\eta) := \omega(\Omega^{-1} \xi, \Omega^{-1} \eta) 
\qquad \forall \xi,\eta\in \mathbb{H}^*,
\]
and $\Omega : (\mathbb{H},\omega) \rightarrow (\mathbb{H}^*,\omega^*)$ is a linear symplectomorphism.
The dual inner product $(\cdot,\cdot)_*$ is compatible with $\omega^*$ and the corresponding complex structure is $-J^*$. The Liouville vector field and the Liouville form on $(\mathbb{H}^*,\omega^*)$ are denoted by $Y^*$ and $\lambda^*$, respectively:
\[
Y^*(\xi) := \frac{\xi}{2}, \qquad 
\lambda^*(\xi)[\eta] := \omega^*(Y^*(\xi),\eta) = \frac{1}{2} \omega^*(\xi,\eta) = -\frac{1}{2} (J^* \xi,\eta)_* \qquad \forall \xi,\eta\in \mathbb{H}^*.
\]
The symplectic action of an absolutely continuous closed curve $\xi:\T \rightarrow \mathbb{H}^*$ is denoted by $\mathbb{A}^*(\xi)$:
\[
\mathbb{A}^*(\xi) := \int_{\T} \xi^*(\lambda^*) = \frac{1}{2} \int_{\T} (J^*\dot{\xi}(t),\xi(t))_*\, dt = \frac{1}{2} \int_{\T} \langle \xi(t), \Omega^{-1} \dot{\xi}(t) \rangle \, dt.
\]

Let $\mathscr{C}$ be the set of all bounded closed convex neighborhoods of the origin in $\mathbb{H}$.
We define a function
\[
a_{\infty} : \mathscr{C} \rightarrow (0,+\infty)
\]
as follows:
\[
a_{\infty}(C) := \sup \set{\mathbb{A}^*(\xi)}{\xi: \T \rightarrow \mathbb{H}^* \mbox{ absolutely continuous with } \dot{\xi} \in C^0 \mbox{ a.e.}},
\]
where $C^0\subset \mathbb{H}^*$ denotes the polar set of $C$, that is
\[
C^0 := \set{\xi \in \mathbb{H}^*}{\langle \xi,x \rangle \leq 1 \; \forall x\in C}.
\]
The fact that $C$ contains a ball of positive radius $r$ about the origin implies that $C^0$ is contained in the ball of radius $1/r$. Therefore, the curves $\xi$ which appear in the definition of $a_{\infty}(C)$ have uniformly bounded derivative, and hence $a_{\infty}(C)$ is finite. Since $C$ is bounded, its polar set $C^0$ contains a ball about the origin. The closed curve 
\[
\xi(t) = e^{-2\pi t J^*} \xi_0, \qquad \xi_0\neq 0,
\]
has positive action $\pi \|\xi_0\|^2_*$ and satisfies $\|\dot{\xi}\|_* = 2\pi \|\xi_0\|_*$. Therefore $\dot{\xi}$ belongs to $C^0$ if $\|\xi_0\|_*$ is small enough, and $a_{\infty}(C)$ is strictly positive. 

We now define 
\[
c_{\mathbb{H}} : \mathscr{C} \rightarrow (0,+\infty) , \qquad
c_{\mathbb{H}}(C) := \frac{1}{4 a_{\infty}(C)}.
\]
We refer to the quantity $c_{\mathbb{H}}(C)$ as to the {\em symplectic capacity of $C$}. The following theorem summarizes the properties of this symplectic capacity which follow directly from the definition. The invariance property is more delicate and we deal with it in Sections \ref{secinv}, \ref{gensymp} and \ref{geninv}.

\begin{thm}
\label{main}
The function $c_{\mathbb{H}}:  \mathscr{C} \rightarrow (0,+\infty)$ satisfies the following properties:
\begin{enumerate}[(i)]
\item {\em (Monotonicity)} If $C_1\subset C_2$ then $c_{\mathbb{H}}(C_1) \leq c_{\mathbb{H}}(C_2)$.
\item {\em (Homogeneity)} $c_{\mathbb{H}} (rC) = r^2 c_{\mathbb{H}}(C)$ for every $r>0$. 
\item {\em (Normalization)} If $B$ is the closed unit ball of $\mathbb{H}$, then $c_{\mathbb{H}} (B) = \pi$.
\item {\em (Projection)} Let $P$ be the symplectic projector onto a symplectic subspace $\mathbb{H}_0\subset \mathbb{H}$. Then $c_{\mathbb{H}_0}(PC) \geq c_{\mathbb{H}}(C)$.
\item {\em (Continuity)} The function $c_{\mathbb{H}} : \mathscr{C} \rightarrow (0,+\infty)$ is continuous with respect to the Hausdorff metric.
\end{enumerate}
\end{thm}

\begin{proof} 
Property (i) is immediate: If $C_1 \subset C_2$, then $C_1^0 \supset C_2^0$, so $a_{\infty}(C_1) \geq a_{\infty}(C_2)$ and hence $c_{\mathbb{H}}(C_1)\leq   c_{\mathbb{H}}(C_2)$. 

Property (ii) follows from the identity
\[
\mathbb{A}^*(r\xi) = r^2 \mathbb{A}^*(\xi),
\]
and from the fact that $(rC)^0 = (1/r) C^0$.

We prove (iii). The polar of the closed unit ball $B$ of $\mathbb{H}$ is the closed unit ball of $\mathbb{H}^*$. The curve 
\[
\xi(t) = e^{-2\pi t J^*} \xi_0
\]
satisfies 
\[
\|\dot{\xi}(t)\|_* = \|-2\pi J^* e^{-2\pi J^* t} \xi_0\|_* = 2\pi \|\xi_0\|_* \qquad \forall t\in \T,
\] 
because $J^*$ and $e^{-2\pi J^*t}$ are isometries. Therefore, $\xi$ is an admissible curve in the definition of $a_{\infty}(B)$ when $\|\xi_0\|_*=1/(2\pi)$ and hence
\[
a_{\infty}(B) \geq \mathbb{A}^*(\xi) = \pi \|\xi_0\|_*^2 = \frac{1}{4 \pi}.
\]
On the other hand, if $\eta:\T \rightarrow \mathbb{H}^*$ is any absolutely continuous closed curve with $\|\dot{\eta}\|_*\leq 1$ a.e. then, denoting by 
\[
\eta_0 := \int_{\T} \eta(t)\, dt
\]
its average, the Cauchy-Schwarz and Poincar\'e inequalities imply
\[
\begin{split}
\mathbb{A}^*(\eta) &= \frac{1}{2} \int_{\T} (J^* \dot{\eta},\eta)_*\, dt = \frac{1}{2} \int_{\T} (J \dot{\eta},\eta-\eta_0)_*\, dt \leq \frac{1}{2} \|\dot{\eta}\|_2 \|\eta-\eta_0\|_2 \\ &\leq \frac{1}{2} \|\dot{\eta}\|_2 \cdot \frac{1}{2\pi} \|\dot{\eta}\|_2 = \frac{1}{4\pi} \|\dot{\eta}\|_2^2 \leq \frac{1}{4\pi},
\end{split}
\]
where $\|\cdot\|_2$ denotes the $L^2$ norm.
We conclude that $a_{\infty}(B)=1/(4\pi)$ and hence $c_{\mathbb{H}}(B)=\pi$.

We prove (iv). We identify $\mathbb{H}_0^*$ with the annihilator of the symplectic orthogonal complement of $\mathbb{H}_0$ (that is, with $\Omega \mathbb{H}_0$). Then the polar of $PC\subset \mathbb{H}_0$ coincides with
\[
\begin{split}
(PC)^0 &= \set{\xi\in \mathbb{H}_0^*}{\langle \xi,x\rangle\leq 1 \; \forall x\in P(C)} \\
&= \set{\xi\in \mathbb{H}_0^*}{\langle\xi,Py\rangle \leq 1\; \forall y\in C} \\
&=\set{\xi\in \mathbb{H}_0^*}{\langle\xi,y\rangle \leq 1\; \forall y\in C} \\
&= C^0 \cap \mathbb{H}_0^*,
\end{split}
\]
where we have used the fact that $\langle \xi,(I-P)y\rangle =0$ because $\xi$ annihilates the symplectic orthogonal of $\mathbb{H}_0$.
Let $\xi:\T \rightarrow \mathbb{H}_0^*$ be a closed absolutely continuous curve as in the definition of $a_{\infty}(PC)$, that is 
\[
\dot{\xi} \in (PC)^0 =C^0 \cap \mathbb{H}_0^*  \quad \mbox{a.e.}.
\]
In particular, $\dot{\xi} \in C^0$ a.e., so $\xi$ is an admissible curve in the definition of $a_{\infty}(C)$ and hence $a_{\infty}(C) \geq a_{\infty}(PC)$. The inequality $c_{\mathbb{H}_0}(PC) \geq c_{\mathbb{H}}(C)$ follows.

Property (v) could be proved directly but is also a consequence of the monotonicity and homogeneity properties. In fact, let $(C_n)\subset \mathscr{C}$ be a sequence which converges to $C\in \mathscr{C}$ in the Hausdorff distance. Fix some $\epsilon>0$. Using the fact that $C$ is a neighbourhood of $0$, the Hausdorff convergence of $(C_n)$ to $C$ implies that 
\[
C_n \subset C + \epsilon C = (1+\epsilon) C 
\]
for $n$ large enough, where the last identity follows from the fact that $C$ is convex. By (i) and (ii) we obtain
\begin{equation}
\label{hau1}
c_{\mathbb{H}}(C_n) \leq (1+\epsilon)^2 c_{\mathbb{H}}(C) = c_{\mathbb{H}}(C) + \epsilon (2+\epsilon) c_{\mathbb{H}}(C) \qquad \mbox{for $n$ large enough}. 
\end{equation}
The fact that $(C_n)$ converges to some neighbourhood of $0$ implies that there are $K_0,K_1\in \mathscr{C}$ such that $K_0 \subset C_n \subset K_1$ for every $n\in \N$. Hence the Haussdorff convergence of $(C_n)$ to $C$ implies that
\[
C \subset C_n + \epsilon K_0 \subset C_n + \epsilon C_n = (1+\epsilon) C_n 
\]
for $n$ large enough. Together with (i) and (ii) we deduce that
\begin{equation}
\label{hau2}
c_{\mathbb{H}}(C) \leq (1+\epsilon)^2 c_{\mathbb{H}}(C_n) \leq  c_{\mathbb{H}}(C_n) + \epsilon (2+\epsilon) c_{\mathbb{H}}(K_1) \qquad \mbox{for $n$ large enough}. 
\end{equation}
Since $\epsilon$ is arbitrary, (\ref{hau1}) and (\ref{hau2}) imply that $(c_{\mathbb{H}}(C_n))$ converges to $c_{\mathbb{H}}(C)$. This proves (v).
\end{proof}

\begin{rem}
It is also easy to show that $c_{\mathbb{H}}$ is invariant with respect to {\em linear} symplectomorphisms. Indeed, this follows from the identity $(\Phi C)^0 = (\Phi^*)^{-1} C^0$, where $\Phi:\mathbb{H} \rightarrow \mathbb{H}$ is a linear isomorphism. On the other hand, the proof of the invariance of $c_{\mathbb{H}}$ with respect to small translations does not seem to be substantially less difficult than the general case that we treat in Sections \ref{secinv}, \ref{gensymp} and \ref{geninv}.
\end{rem}

\begin{ex}
Let $(\mathbb{H},(\cdot,\cdot))$ be a complex Hilbert space, endowed with its standard symplectic structure (see Example \ref{complex}; here we drop the subscript $\C$ from the notation). Let $A: \mathbb{H} \rightarrow \mathbb{H}$ be a (complex linear) bounded self-adjoint operator with spectrum $\sigma(A)$ contained in $(0,+\infty)$, and denote by
\[
\rho(A) := \max \sigma(A)
\]
the spectral radius of $A$, which by self-adjointness coincides with the operator norm of $A$. The ellipsoid
\[
E := \set{x\in \mathbb{H}}{(Ax,x) \leq 1}
\]
is an element of $\mathscr{C}$, and we want to show that its symplectic capacity is
\[
c_{\mathbb{H}}(E) = \frac{\pi}{\rho(A)}.
\]
Here it is convenient to identify the dual of $\mathbb{H}$ with $\mathbb{H}$ itself by the Hermitian product. With this identification, the polar of $E$ is the ellipsoid
\[
E^{\circ} = \set{x\in \mathbb{H}}{(A^{-1}x,x) \leq 1}.
\]
Let $\epsilon>0$. Since $A$ is self-adjoint, we have
\[
\inf_{\|x\|=1} (A^{-1} x,x) = \min \sigma(A^{-1}) = \min \frac{1}{\sigma(A)} = \frac{1}{\max \sigma(A)} = \frac{1}{\rho(A)},
\]
so we can find a unit vector $x_0\in \mathbb{H}$ such that
\[
(A^{-1} x_0 ,x_0) \leq \frac{1}{\rho(A)} (1+\epsilon).
\]
Let $r>0$. The closed curve $x(t) = r e^{2\pi i t} x_0$ satisfies
\[
(A^{-1} \dot{x}(t), \dot{x}(t)) = 4 \pi^2 r^2 (A^{-1} x_0,x_0) \leq  \frac{4 \pi^2 r^2}{\rho(A)} (1+\epsilon),
\]
so $\dot{x}$ takes values into $E^0$ if
\[
r^2 = \frac{\rho(A)}{4\pi^2(1+\epsilon)}.
\]
With such a choice of $r$, we obtain
\[
a_{\infty}(E) \geq \mathbb{A}(x) = \pi r^2 = \frac{\rho(A)}{4\pi(1+\epsilon)}.
\]
Since $\epsilon$ is arbitrary, we deduce that
\[
a_{\infty}(E) \geq \frac{\rho(A)}{4\pi},
\]
and hence
\[
c_{\mathbb{H}}(E) = \frac{1}{4 a_{\infty}(E)} \leq \frac{\pi}{\rho(A)}.
\]
On the other hand, since
\[
(Ay,y) \leq \|A\| \|y\|^2 = \rho(A) \|y\|^2 \qquad \forall y\in \mathbb{H},
\]
the ball of radius $\rho(A)^{-1/2}$ is contained in $E$, and by the monotonicity property of $c_{\mathbb{H}}$ we deduce the opposite inequality:
\[
c_{\mathbb{H}}(E) \geq \frac{\pi}{\rho(A)}.
\]
\end{ex}

\section{Equivalent formulations and closed characteristics}
\label{eqsec}

Let $C\subset \mathbb{H}$ be a closed bounded convex neighborhood of the origin. We denote by 
\[
\mu_C: \mathbb{H} \rightarrow [0,+\infty)
\] 
its Minkowski gauge, that is the positively 1-homogeneous convex functional
\[
\mu_C(x) := \inf\set{t\geq 0}{x\in tC}.
\]
With this notation we have
\[
a_{\infty}(C) = \sup \set{ \mathbb{A}^*(\xi)}{\xi: \T \rightarrow \mathbb{H}^* \mbox{ absolutely continuous, } \| \mu_{C^0}(\dot{\xi})\|_{\infty} \leq 1}.
\]
It is useful to relax the condition which defines the set of curves appearing in the above set: For $1\leq p < \infty$ we set
\[
a_p(C) := \sup \set{ \mathbb{A}^*(\xi)}{\xi: \T \rightarrow \mathbb{H}^* \mbox{ absolutely continuous, } \| \mu_{C^0}(\dot{\xi})\|_p \leq 1},
\]
where $\|\cdot\|_p$ denotes the $L^p$ norm. The following proposition provides us with some useful alternative formulas for the symplectic capacity of a convex set. See \cite[Proposition 2.1]{ao08} for similar computations in a finite-dimensional setting.

\begin{prop}
\label{cara} Let $C$ be an element of $\mathscr{C}$. Then the numbers $a_p(C)$ coincide for every $p\in [1,\infty]$. In particular,
\[
c_{\mathbb{H}}(C) = \frac{1}{4a_p(C)} \qquad \forall p\in [1,\infty].
\]
Furthermore
\begin{equation}
\label{infi}
c_{\mathbb{H}}(C) = \frac{1}{4} \inf \set{\|\mu_{C^0}(\dot{\xi})\|_p^2}{\xi: \T \rightarrow \mathbb{H}^* \mbox{ absolutely continuous, }\mathbb{A}^*(\xi)=1},
\end{equation}
for every $p\in [1,\infty]$.
\end{prop}  

\begin{proof}
Let $\xi:\T \rightarrow \mathbb{H}^*$ be an absolutely continuous closed curve. By the H\"older inequality and the fact that $\T$ has measure 1, the function
\[
p \mapsto \|\mu_{C^0}(\dot{\xi})\|_p,
\]
is increasing on $[1,+\infty]$, so there holds
\[
a_p(C) \geq a_q(C) \qquad \mbox{if } p,q\in [1,\infty] \mbox{ satisfy } p<q. 
\]
Therefore, it is enough to show that $a_1(C) \leq a_{\infty}(C)$. For every $\epsilon>0$ we can find a smooth closed curve $\xi:\T \rightarrow \mathbb{H}$ such that $\dot{\xi}(t)\neq 0$ for every $t\in \T$,
\[
\|\mu_{C^0}(\dot{\xi})\|_1 \leq 1,
\]
and
\[
\mathbb{A}^*(\xi)> a_1(C) - \epsilon.
\]
Let $\tau:\R \rightarrow \R$ be the solution of the Cauchy problem
\[
\tau'(s) = \frac{1}{\mu_{C^0}(\dot{\xi}(\tau(s)))} , \qquad \tau(0)=0.
\]
Then $\tau$ is strictly increasing, and we denote by $\sigma$ its inverse. Then
\[
\sigma(1) = \int_0^1 \sigma'(t)\, dt = \int_0^1 \frac{1}{\tau'(\sigma(t))}\, dt = \int_0^1 \mu_{C^0}(\dot{\xi}(t))\, dt \leq 1.
\] 
The absolutely continuous curve
\[
\eta: [0,1] \rightarrow \mathbb{H}^*, \qquad \eta(s) := \left\{ \begin{array}{ll} \xi(\tau(s)) & \mbox{for } 0\leq s \leq \sigma(1), \\ \xi(0)  & \mbox{for } \sigma(1) < s \leq 1, \end{array}  \right.
\]
is closed, satisfies
\[
\mathbb{A}^*(\eta)= \mathbb{A}^*(\xi) > a_1(C) - \epsilon,
\]
and
\[
\mu_{C^0}(\dot{\eta}(s) ) = \left\{ \begin{array}{cl} \tau'(s) \mu_{C_0}(\dot{\xi}(\tau(s))) = 1 & \mbox{for } 0\leq s < \sigma(1), \\ 0 & \mbox{for } \sigma(1)<s<1. \end{array} \right.
\]
Therefore, $\eta$ is an admissible curve in the definition of $a_{\infty}(C)$ and hence
\[
a_{\infty}(C) \geq \mathbb{A}^*(\eta) > a_1(C) - \epsilon.
\] 
We conclude that $a_{\infty}(C) \geq a_1(C)$ and thus all the $a_p(C)$'s coincide.

Furthermore we have, using the 2-homogeneity of $\mathbb{A}^*$, 
\[
\begin{split}
a_p(C) &= \sup_{\mathbb{A}^*(\xi)>0} \frac{\mathbb{A}^*(\xi)}{\|\mu_{C^0} (\dot{\xi})\|_p^2} = \left( \inf_{\mathbb{A}^*(\xi)>0} \frac{\|\mu_{C^0}(\dot{\xi})\|_p^2}{\mathbb{A}^*(\xi)} \right)^{-1} \\ &= \Bigl( \inf \set{\|\mu_{C^0}(\dot{\xi})\|_p^2}{\mathbb{A}^*(\xi)=1} \Bigr)^{-1}.
\end{split}
\]
Therefore,
\[
c_{\mathbb{H}}(C) = \frac{1}{4a_p(C)} =\frac{1}{4} \inf \set{\|\mu_{C^0}(\dot{\xi})\|_p^2}{\mathbb{A}^*(\xi)=1},
\]
as claimed.
\end{proof}

The next result relates the symplectic capacity $c_{\mathbb{H}}(C)$ to the minimal action of closed characteristics on the boundary of $C$, provided that such closed characteristics of minimal action exist. Recall that closed characteristics on the boundary of a convex set have always positive action, as shown at the end of Section \ref{symchar}.

\begin{thm}
\label{minicara}
Assume that the boundary of $C\in \mathscr{C}$ is of class $C^1$. If $x:\T \rightarrow \partial C$ is a closed characteristic on $\partial C$ then
\[
\mathbb{A}(x) \geq c_{\mathbb{H}}(C).
\]
Let $p\in [1,\infty]$.
If $\xi:\T \rightarrow \mathbb{H}^*$ is an absolutely continuous curve with $\|\mu_{C^0} (\dot{\xi})\|_p \leq 1$ (resp.\ with $\A^*(\xi)=1$) which achieves the supremum which defines $a_p(C)$ (resp.\ the infimum (\ref{infi})), then $-\Omega^{-1} \xi$ is homothetic to a closed characteristic $x$ on $\partial C$ such that
\[
\mathbb{A}(x) = c_{\mathbb{H}}(C). 
\]
When $\mathbb{H}$ is finite dimensional, the supremum which defines $a_p(C)$ (resp.\ the infimum (\ref{infi})) is a maximum (resp.\ a minimum), and hence $c_{\mathbb{H}}(C)$ coincides with the minimal action over all closed characteristics on $\partial C$. 
\end{thm}

In particular, the above result shows that $c_{\mathbb{H}}$ coincides with the Ekeland-Hofer and with the Hofer-Zehnder symplectic capacity of $C$ when $\mathbb{H}$ is finite dimensional (see \cite[Proposition 3.10]{vit89} and \cite[Proposition 4]{hof90}). The proof of the above result uses standard arguments from Clarke's duality. For sake of completeness we review it in Section \ref{last}.

When $\mathbb{H}$ is infinite dimensional, the supremum  which defines $a_p(C)$ may not be achieved. Actually, $\partial C$ may contain no closed characteristics, as the example of the ellipsoid which is described in the first part of Section \ref{example} shows. 

On the other extreme, when $\mathbb{H}$ has dimension two and $C\in \mathscr{C}$
has $C^1$ boundary, $\partial C$ itself is the unique closed characteristic on $\partial C$ and has action $\mathrm{area}_{\omega}(C)$. Therefore, the last assertion of Theorem \ref{minicara} implies that in this case: 
\begin{equation}
\label{dim2}
c_{\mathbb{H}}(C) = \mathrm{area}_{\omega}(C).
\end{equation}
Since both $c_{\mathbb{H}}$ and the area are continuous with respect to the Hausdorff metric on $\mathscr{C}$, the above formula holds for an arbitrary $C\in \mathscr{C}$.

\begin{rem}
It is possible to define characteristics on $\partial C$ also when $\partial C$ is not $C^1$: A characteristic is an absolutely continuous curve on $\partial C$ whose derivative belongs a.e.\ to the normal cone of $C$ rotated by $J$. See \cite{cla81} and \cite[Sections II.4 and V.1]{eke90}. Using this definition, Theorem \ref{minicara} would hold for an arbitrary $C\in \mathscr{C}$. In finite dimensions, the theory of closed characteristics on non-smooth boundaries of convex sets has beautiful applications to the dynamics of convex billiards, see \cite{ao14}.
\end{rem}

\section{Invariance with respect to homogeneous symplectomorphisms}
\label{secinv}
In this section we begin the proof of the invariance property of $c_{\mathbb{H}}$, by considering a special class of symplectomorphisms.

A $C^1$ symplectomorphism $\varphi: \mathbb{H} \setminus \{0\} \rightarrow \mathbb{H} \setminus \{0\}$ is said to be {\em homogeneous} if it is a positively 1-homogeneous map:
\[
\varphi(tx) = t \, \varphi(x) \qquad \forall x\in \mathbb{H}, \; \forall t>0.
\]
Homogeneous symplectomorphisms can be extended continuously in $0$ by setting $\varphi(0)=0$, but such an extension is in general not differentiable at $0$. More about homogeneous symplectomorphism will be said in Section \ref{gensymp}.

If $C\in \mathscr{C}$ and $\mu_C$ is its Minkowski gauge, we set
\[
H_C := \frac{1}{2} \mu_C^2.
\]
The function $H_C$ is convex and positively 2-homogeneous. Moreover, the Fenchel conjugate $H_C^*$ of $H_C$ is the function $H_{C^0}:\mathbb{H}^* \rightarrow \R$. Indeed, this follows from the fact that $\mu_{C^0}$ coincides with the support function of $C$, 
\[
\mu_{C^0} (\xi ) = \sup_{x\in C} \langle \xi, x\rangle = \sup_{\mu_C(x)=1} \langle \xi, x\rangle \qquad \forall \xi\in \mathbb{H}^*,
\]
and from the computation
\[
\begin{split}
H_C^*(\xi) &:= \sup_{x\in \mathbb{H}} \left( \langle \xi,x \rangle - H_C(x) \right) =
\sup_{r\geq 0} \sup_{\mu_C(x)=1} \left( \langle \xi, r x \rangle - \frac{1}{2} \mu_C(r x)^2 \right) \\ &= \sup_{r\geq 0} \left( r \sup_{\mu_C(x)=1} \langle \xi,x \rangle - \frac{r^2}{2} \right) = \sup_{r\geq 0} \left( r \mu_{C^0}(\xi) - \frac{r^2}{2} \right) = \frac{\mu_{C^0}(\xi)^2}{2} = H_{C^0} (\xi).
\end{split} 
\]

We denote by $\widehat{\mathscr{C}}$ the subset of $\mathscr{C}$ consisting of those $C$ in $\mathscr{C}$ for which $H_C$ and $H_{C^0}$ are continuously differentiable and the maps 
\[
dH_C : \mathbb{H} \rightarrow \mathbb{H}^* \qquad \mbox{and} \qquad
dH_{C^0} : \mathbb{H}^*\rightarrow \mathbb{H}
\]
are globally Lipschitz continuous. These are regularity and strong convexity assumptions on $C$. Indeed, $dH_{C^0} = dH_C^*$ is $c$-Lipschitz if and only if $H_C$ is {\em $c$-strongly convex}, meaning that the function
\[
x\mapsto H_C(x) - \frac{1}{2c} \|x\|^2
\]
is convex (see e.g.\ \cite[Theorem 18.15]{bc11}).

Let $C\in \widehat{\mathscr{C}}$. Since $H_{C^0}$ is the Fenchel conjugate of $H_C$, the Legendre reciprocity formula gives us
\begin{equation}
\label{leg1}
dH_{C^0} = dH_C^{-1}.
\end{equation}
See e.g.\ \cite[Proposition II.1.15]{eke90}. Moreover,
\begin{equation}
\label{leg2}
H_{C^0} \circ dH_C = H_C \qquad \mbox{and} \qquad H_C \circ dH_{C^0} = H_{C^0}.
\end{equation}
Indeed, the Euler identity for the 2-homogeneous function $H_{C^0}$ gives 
\[
2 H_{C^0}(\xi) = \langle \xi, dH_{C^0}(\xi) \rangle \qquad \forall \xi \in \mathbb{H}^*,
\]
and by choosing $\xi= dH_C(x)$ with $x\in \mathbb{H}$, the identity (\ref{leg1}) and the Euler identity for the 2-homogeneous function $H_C$ imply
\[
2 H_{C^0}(dH_C(x)) = \langle dH_C(x), dH_{C^0}(dH_C(x)) \rangle = \langle dH_C(x),x \rangle = 2 H_C(x),
\]
proving the first identity of (\ref{leg2}). The second one follows by exchanging the role of $C$ and $C^0$, since the polar is involutive.

The aim of this section is to prove that the restriction of the capacity $c_{\mathbb{H}}$ to the set $\widehat{\mathscr{C}}$ is invariant with respect to homogeneous symplectomorphisms: 

\begin{thm}
\label{invariance}
Let $\psi: \mathbb{H}\setminus \{0\} \rightarrow \mathbb{H}\setminus \{0\}$ be a $C^1$ homogeneous symplectomorphism, continuously extended at the origin by setting $\psi(0):=0$,  such that the 0-homogeneous maps $d\psi$ and $d\psi^{-1}$ are bounded. Assume that $C$ and $\psi(C)$ belong to $\widehat{\mathscr{C}}$. Then
\begin{equation}
\label{capid}
c_{\mathbb{H}}(\psi(C)) = c_{\mathbb{H}}(C).
\end{equation}
\end{thm}

Let $C$ and $\psi(C)$ be elements of $\widehat{\mathscr{C}}$.
In the proof of (\ref{capid}) we shall use the following formula for the symplectic capacity of $C$, which is a consequence of Proposition \ref{cara}:
\[
c_{\mathbb{H}}(C) = \inf \set{ \frac{1}{2} \int_{\T} H_{C^0}(\dot{\xi})\, dt}{\xi:\T \rightarrow \mathbb{H}^* \mbox{ absolutely continuous, } \mathbb{A}^*(\xi)=1}.
\]
In other words, $c_{\mathbb{H}}(C)$ is the infimum of the restriction of the functional
\[
\Phi_C : H^1(\T,\mathbb{H}^*) \rightarrow \R, \qquad \Phi_C(\xi) := \frac{1}{2} \int_{\T} H_{C^0}(\dot{\xi})\, dt,
\]
to the closed subset
\[
\mathbb{M} := \set{\xi\in H^1(\T,\mathbb{H}^*)}{\mathbb{A}^*(\xi)=1}
\]
of the Sobolev space $H^1(\T,\mathbb{H}^*)$ of absolutely continuous closed curves in $\mathbb{H}^*$ whose derivative has square integrable norm.
Since $H_{C^0}$ is continuously differentiable and $dH_{C^0}$ is globally Lipschitz, the functional $\Phi_C$ is continuously differentiable on $H^1(\T,\mathbb{H}^*)$, and its differential is
\begin{equation}
\label{difPhi}
d\Phi_C(\xi)[\eta] = \frac{1}{2} \int_{\T} \langle \dot{\eta}, dH_{C^0}(\dot{\xi}) \rangle \, dt,  \qquad \forall \xi,\eta\in H^1(\T,\mathbb{H}^*).
\end{equation}
On the other hand, 1 is a regular value of the smooth functional $\mathbb{A}^*: H^1(\T,\mathbb{H}^*) \rightarrow \R$, so $\mathbb{M}$ is a smooth submanifold of $H^1(\T,\mathbb{H}^*)$. It is a complete Riemannian manifold with the metric induced from $H^1(\T,\mathbb{H}^*)$. These facts allow us to find a sequence $(\xi_n)\subset \mathbb{M}$ such that
\begin{equation}
\label{mini}
\Phi_C(\xi_n) \rightarrow c_{\mathbb{H}}(C)
\end{equation}
and such that $(\xi_n)$ is a Palais-Smale sequence for $\Phi_C|_{\mathbb{M}}$, meaning that
\begin{equation}
\label{ps}
d(\Phi_C|_{\mathbb{M}})(\xi_n)  \rightarrow 0
\end{equation}
with respect to the induced metric on $T^* \mathbb{M}$. This follows from the standard deformation lemma (see e.g. \cite[Lemma I.3.2]{cha93}): If no sequence satisfying (\ref{mini}) and (\ref{ps}) exists, then, by using the negative gradient flow of $\Phi_C|_{\mathbb{M}}$, it is possible to deform the non-empty sublevel $\set{\xi\in \mathbb{M}}{\Phi_C(\xi)<c_{\mathbb{H}}(C) + \epsilon}$ into the empty sublevel $\set{\xi\in \mathbb{M}}{\Phi_C(\xi)<c_{\mathbb{H}}(C) - \epsilon}$ when $\epsilon>0$ is sufficiently small, which is a contradiction.

The next lemma says that, up to the addiction of suitable constants, the closed curves $-\Omega^{-1} \xi_n:\T \rightarrow \mathbb{H}$ are ``almost solutions'' of the Hamiltonian equation
\[
-\Omega \dot{x} = 2c_{\mathbb{H}}(C) dH_C(x).
\]

\begin{lem}
\label{lll1}
Assume that the sequence $(\xi_n)\subset \mathbb{M}$ satisfies (\ref{mini}) and (\ref{ps}). Then there exists a sequence $(\bar{y}_n)$ of constant loops in $\mathbb{H}$ such that the loops $y_n:= -\Omega^{-1} \xi_n + \bar{y}_n$ are uniformly bounded in $H^1(\T,\mathbb{H})$ and satisfy
\begin{equation}
\label{hs}
-\Omega \dot{y}_n = 2c_{\mathbb{H}}(C) dH_C(y_n) + \eta_n,
\end{equation}
where $(\eta_n)$ is infinitesimal in $L^2(\T,\mathbb{H}^*)$.
\end{lem}

\begin{proof}
Condition (\ref{ps}) can be rewritten as
\begin{equation}
\label{ps2}
d\Phi_C(\xi_n) = \lambda_n d\mathbb{A}^*(\xi_n) + Z_n,
\end{equation}
where $(\lambda_n)\subset \R$ and  $(Z_n)$ is an infinitesimal sequence in the dual space of $H^1(\T,\mathbb{H}^*)$. Since both $d\Phi_C(\xi_n)$ and $d\mathbb{A}^*(\xi_n)$ vanish on constant loops, so does $Z_n$. Therefore, $Z_n$ has the form
\begin{equation}
\label{zeta}
Z_n (\eta) = \int_{\T} \langle z_n,\dot{\eta}\rangle \, dt, \qquad \forall \eta\in H^1(\T,\mathbb{H}^*),
\end{equation}
where $(z_n)$ is an infinitesimal sequence in $L^2(\T,\mathbb{H})$. By (\ref{mini}), $(\dot{\xi}_n)$ is bounded in $L^2$. In particular, the sequence $(Z_n(\xi_n))$ is infinitesimal. Since both $\Phi_C$ and $\mathbb{A}^*$ are positively 2-homogeneous, (\ref{ps2}) and the Euler identity imply
\[
2 \Phi_C(\xi_n) = d\Phi_C(\xi_n)[\xi_n] = \lambda_n d\mathbb{A}^*(\xi_n)[\xi_n] + Z_n(\xi_n) = 2 \lambda_n \mathbb{A}^*(\xi_n) + o(1) = 2 \lambda_n + o(1),
\]
where $o(1)$ denotes an infinitesimal sequence. Together with (\ref{mini}), we deduce that
\begin{equation}
\label{lambda}
\lambda_n = \Phi_C(\xi_n) + o(1) = c_{\mathbb{H}}(C) + o(1).
\end{equation}
The differential of the functional $\mathbb{A}^*$ on $H^1(\T,\mathbb{H}^*)$ has the form
\begin{equation}
\label{difA}
d\mathbb{A}^*(\xi)[\eta] = \int_{\T} \langle \xi,\Omega^{-1} \dot{\eta} \rangle\, dt = - \int_{\T} \langle \dot{\eta},\Omega^{-1} \xi \rangle\, dt,  \qquad \forall \xi,\eta\in H^1(\T,\mathbb{H}^*).
\end{equation}
Formulas (\ref{difPhi}), (\ref{zeta}) and (\ref{difA}) allow us to rewrite (\ref{ps2}) as
\[
\int_{\T} \Bigl\langle \dot{\eta}, \frac{1}{2} dH_{C^0}(\dot{\xi}_n) + \lambda_n \Omega^{-1} \xi_n -z_n  \Bigr\rangle \, dt = 0, \qquad \forall \eta\in H^1(\T,\mathbb{H}).
\]
By the Du Bois-Reymond Lemma the above fact implies that for every $n\in \N$ the curve 
\[
\frac{1}{2} dH_{C^0}(\dot{\xi}_n) + \lambda_n \Omega^{-1} \xi_n -z_n
\]
is a.e.\ constant on $\T$. Therefore, there exists a sequence $(\bar{y}_n)$ of constant loops in $\mathbb{H}$ such that
\[
dH_{C^0}(\dot{\xi}_n) = 2 \lambda_n (- \Omega^{-1} \xi_n + \bar{y}_n)  + 2 z_n,
\]
for every $n\in \N$. By the above identity, the sequence $y_n:=-\Omega^{-1} \xi_n + \bar{y}_n $ is bounded in $L^2$. Since $\dot{y}_n=-\Omega^{-1}\dot{\xi}_n$ is also uniformly bounded in $L^2$, the sequence $(y_n)$ is bounded in $H^1(\T,\mathbb{H})$. Moreover,
the above identity can be rewritten by using (\ref{lambda}) as
\[
dH_{C^0}(-\Omega \dot{y}_n) = 2 c_{\mathbb{H}}(C) y_n + v_n,
\]
where $(v_n)$ is infinitesimal in $L^2(\T,\mathbb{H})$. By applying $dH_C$ to both sides we get by the Legendre reciprocity formula (\ref{leg1}):
\[
-\Omega \dot{y}_n = dH_C(2c_{\mathbb{H}}(C) y_n + v_n).
\]
Since $dH_C$ is positively 1-homogeneous and globally Lipschitz continuous, we deduce that
\[
-\Omega \dot{y}_n = 2c_{\mathbb{H}}(C) dH_C(y_n) + \eta_n,
\]
where $(\eta_n)$ is infinitesimal in $L^2(\T,\mathbb{H}^*)$.
\end{proof}

The next lemma says, in particular, that the Hamiltonian function $H_C$ is ``almost constant'' along the closed curve $y_n$:

\begin{lem}
\label{lll2}
Let $(y_n)$ be the sequence given by Lemma \ref{lll1}. Then the sequence of real functions $(H_C\circ y_n)$ converges uniformly to the constant function $1/(2c_{\mathbb{H}}(C))$. 
\end{lem}

\begin{proof}
By differentiating $H_C \circ y_n$ we find by (\ref{hs})
\[
\frac{d}{dt} H_C\circ y_n = \langle dH_C(y_n),\dot{y}_n \rangle = \frac{1}{2c_{\mathbb{H}}(C)} \langle - \Omega \dot{y}_n - \eta_n, \dot{y}_n \rangle = -  \frac{1}{2c_{\mathbb{H}}(C)} \langle \eta_n, \dot{y}_n \rangle,
\]
because $\Omega$ is skew-symmetric. Since $(\dot{y}_n)$ is bounded in $L^2$ and $(\eta_n)$ is infinitesimal in $L^2$, the sequence 
\[
\left( \frac{d}{dt} H_C\circ y_n \right)
\]
is infinitesimal in $L^1$. Together with the fact that $(H_C\circ y_n)$ is bounded in $L^{\infty}$,  because so is $(y_n)$, we deduce that the sequence $(H_C\circ y_n)$ converges uniformly to a constant function $m$. We must show that $m=1/(2c_{\mathbb{H}}(C))$.

From (\ref{mini}) and (\ref{hs}) we deduce
\begin{equation}
\label{passi}
\begin{split}
c_{\mathbb{H}}(C) + o(1) &= \Phi_C(\xi_n) = \frac{1}{2} \int_{\T} H_{C^0}(\dot{\xi}_n)\, dt = \frac{1}{2} \int_{\T} H_{C^0}(-\Omega \dot{y}_n)\, dt \\ &= \frac{1}{2} \int_{\T} H_{C^0}(2c_{\mathbb{H}}(C) dH_C(y_n) + \eta_n)\, dt.
\end{split}
\end{equation}
Since $dH_{C^0}$ is globally Lipschitz and vanishes at the origin, we have the pointwise estimate
\[
\begin{split}
|H_{C^0}(2c_{\mathbb{H}}(C) dH_C(y_n) &+ \eta_n) - H_{C^0}(2c_{\mathbb{H}}(C) dH_C(y_n))| \\ &\leq \|\eta_n\| \sup_{\theta\in (0,1)} \|dH_{C^0}(2c_{\mathbb{H}}(C) dH_C(y_n) + \theta \eta_n)\|  \\ &\leq M \bigl( 2c_{\mathbb{H}}(C) \|dH_C(y_n)\| + \|\eta_n\|\bigr) \|\eta_n\|,
\end{split}
\]
where $M$ is the Lipschitz  constant of $dH_{C^0}$. Since $(dH_C(y_n))$ is bounded in $L^{\infty}$, because so is $(y_n)$, we obtain a pointwise estimate of the form
\[
|H_{C^0}(2c_{\mathbb{H}}(C) dH_C(y_n) + \eta_n) - H_{C^0}(2c_{\mathbb{H}}(C) dH_C(y_n))| \leq M' (1+\|\eta_n\|) \|\eta_n\|,
\]
for some constant $M'$. Since $(\eta_n)$ is infinitesimal in $L^2$, the above estimate implies that
\begin{equation}
\label{fine}
\begin{split}
\frac{1}{2} \int_{\T} H_{C^0}(2c_{\mathbb{H}}(C) dH_C(y_n) + \eta_n)\, dt &= \frac{1}{2} \int_{\T} H_{C^0}(2c_{\mathbb{H}}(C) dH_C(y_n))\, dt + o(1) \\ &= 2 c_{\mathbb{H}}(C)^2 \int_{\T} H_{C^0}( dH_C(y_n))\, dt + o(1) \\ &= 2 c_{\mathbb{H}}(C)^2 \int_{\T} H_C(y_n)\, dt + o(1),
\end{split}
\end{equation}
where we have used also the 2-homogeneity of $H_{C^0}$ and the first identity in (\ref{leg2}). Since $H_C\circ y_n$ converges uniformly to  the constant $m$, (\ref{passi}) and (\ref{fine}) imply
\[
c_{\mathbb{H}}(C) + o(1) = 2 c_{\mathbb{H}}(C)^2  \cdot m + o(1).
\]
The above identity implies that $m=1/(2c_{\mathbb{H}}(C))$.
\end{proof}

We can finally prove Theorem \ref{invariance}.

\begin{proof}[Proof of Theorem \ref{invariance}] Since $\psi$ is 1-homogeneous, the functions $H_C$ and $H_{\psi(C)}$ are related by the identity
\begin{equation}
\label{hom}
H_{\psi(C)} = H_C \circ \psi^{-1}.
\end{equation}
Differentiating this identity and evaluating at $\psi(y)$, $y\neq 0$, we obtain
\begin{equation}
\label{diff}
dH_{\psi(C)}(\psi(y))  = dH_C(y) \circ d\psi^{-1}(\psi(y)) = dH_C(y) \circ d\psi(y)^{-1}, \qquad \forall y\in \mathbb{H}\setminus \{0\}.
\end{equation}
The fact that $\psi$ is symplectic implies that
\[
d\psi(y)^* \, \Omega \, d\psi(y) = \Omega, \qquad \forall y\in \mathbb{H}\setminus \{0\},
\]
and hence 
\begin{equation}
\label{symp}
\Omega \, d\psi(y) = (d\psi(y)^*)^{-1} \, \Omega,
\qquad \forall y\in \mathbb{H}\setminus \{0\}.
\end{equation}
Set $z_n := \psi\circ y_n$ and $\zeta_n := -\Omega z_n$. Notice that by Lemma \ref{lll2} the closed curves $y_n$ and $z_n$ do not pass from the origin when $n$ is large enough. Since $-\Omega$ and $\psi$ are symplectic,
\[
\mathbb{A}^*(\zeta_n) =
\mathbb{A}(z_n) = \mathbb{A}(y_n) = \mathbb{A}^*(-\Omega y_n) = \mathbb{A}^*(\xi_n) = 1,
\]
so $\zeta_n$ belongs to $\mathbb{M}$.
By differentiating $z_n$ we find, by (\ref{symp}) and (\ref{hs}) and by the fact that $d\psi^{-1}$ is uniformly bounded on $\mathbb{H} \setminus \{0\}$
\[
\begin{split}
-\Omega \, \dot{z}_n &= -\Omega\, d\psi(y_n)[\dot{y}_n] = -(d\psi(y_n)^*)^{-1} [\Omega \, \dot{y}_n) ] = (d\psi(y_n)^*)^{-1} [2c_{\mathbb{H}}(C) dH_C(y_n) + \eta_n] \\ &= 2c_{\mathbb{H}}(C) (d\psi(y_n)^*)^{-1} [dH_C(y_n)] + \theta_n,
\end{split}
\]
where $(\theta_n)$ is infinitesimal in $L^2$. From the tautological identity
\[
(d\psi(y)^*)^{-1} [dH_C(y)] = dH_C(y) \circ d\psi(y)^{-1}, \qquad \forall y\in \mathbb{H} \setminus \{0\},
\]
and from (\ref{diff}) we conclude that
\[
-\Omega \, \dot{z}_n = 2c_{\mathbb{H}}(C) dH_{\psi(C)}(z_n) + \theta_n.
\]
Next we compute
\[
\begin{split}
\Phi_{\psi(C)}(\zeta_n) &= \frac{1}{2} \int_{\T} H_{\psi(C)^{\circ}}(\dot{\zeta}_n)\, dt = \frac{1}{2} \int_{\T} H_{\psi(C)^{\circ}}(-\Omega \dot{z}_n)\, dt \\ &=  \frac{1}{2} \int_{\T} H_{\psi(C)^{\circ}}(2c_{\mathbb{H}}(C) dH_{\psi(C)}(z_n) + \theta_n)\, dt.
\end{split}
\]
Arguing as in the proof of Lemma \ref{lll2}, using the fact that $(\theta_n)$ is infinitesimal in $L^2$, that $(dH_{\psi(C)}(z_n))$ is bounded in $L^{\infty}$ and that $dH_{\psi(C)^0}$ is globally Lipschitz, we deduce that the above integral differs from
\[
\frac{1}{2} \int_{\T} H_{\psi(C)^{\circ}}(2c_{\mathbb{H}}(C) dH_{\psi(C)}(z_n))\, dt
\]
by an infinitesimal sequence. Therefore, using also the 2-homogeneity of $H_{\psi(C)^{\circ}}$ and the first identity in (\ref{leg2}), we obtain
\[
\Phi_{\psi(C)}(\zeta_n) = 2 c_{\mathbb{H}}(C)^2 \int_{\T} H_{\psi(C)^{\circ}}(dH_{\psi(C)}(z_n))\, dt + o(1) = 2 c_{\mathbb{H}}(C)^2 \int_{\T} H_{\psi(C)}(z_n)\, dt + o(1).
\]
From (\ref{hom}) we deduce that
\[
\Phi_{\psi(C)}(\zeta_n) = 2 c_{\mathbb{H}}(C)^2 \int_{\T} H_{C}(y_n)\, dt + o(1),
\]
and by Lemma \ref{lll2} the last integral converges to $1/(2c_{\mathbb{H}}(C))$. We conclude that
\[
\Phi_{\psi(C)}(\zeta_n) = c_{\mathbb{H}}(C) + o(1),
\]
which implies that
\[
c_{\mathbb{H}}(\psi(C)) = \inf_{\xi\in \mathbb{M}} \Phi_{\psi(C)} (\xi) \leq c_{\mathbb{H}}(C).
\]
By applying what we have just proved to the convex set $\psi(C)$ and to the homogeneous symplectomorphism $\psi^{-1}$ we obtain the opposite inequality. Therefore,
\[
c_{\mathbb{H}}(\psi(C)) =  c_{\mathbb{H}}(C).
\]
\end{proof}

\section{More facts about homogeneous symplectomorphisms}
\label{gensymp}

Smooth homogeneous symplectomorphisms have the following simple characterization in terms of the Liouville form $\lambda$:

\begin{lem}
Let $\varphi:\mathbb{H}\setminus \{0\} \rightarrow \mathbb{H}\setminus \{0\}$ be a smooth diffeomorphism. Then the following statements are equivalent:
\begin{enumerate}[(i)]
\item $\varphi$ is symplectic and positively 1-homogeneous;
\item $\varphi^* \lambda = \lambda$ on $\mathbb{H}\setminus \{0\}$.
\end{enumerate}
\end{lem}

\begin{proof}
The diffeomorphism $\varphi$ maps rays through the origin into rays through the origin if and only if $\varphi^* Y = f Y$, where $Y(x)=x/2$ is the Liouville vector field and $f$ is a non-vanishing scalar smooth function. Since the differential of a 1-homogeneous map is 0-homogeneous, we deduce that $\varphi$ is positively 1-homogeneous if and only if $\varphi^* Y = f Y$ with $f$ positively 0-homogeneous. 

Assume that $\varphi$ satisfies (ii). Then
\[
\varphi^* \omega = \varphi^* d\lambda = d \varphi^* \lambda = d \lambda = \omega,
\]
so $\varphi$ is symplectic. Moreover, the identity
\[
\imath_{\varphi^* Y} \omega = \imath_{\varphi^* Y} \varphi^* \omega = \varphi^* \imath_Y \omega = \varphi^* \lambda = \lambda = \imath_Y \omega
\]
and the non-degeneracy of $\omega$ imply that $\varphi^* Y = Y$. In particular, $\varphi$ is positively 1-homogeneous. We conclude that $\varphi$ satisfies (i).

Next assume that $\varphi$ satisfies (i) and let $f$ be the positively 0-homogeneous function such that $\varphi^*Y = fY$. Then
\[
\varphi^* \lambda = \varphi^* \imath_Y \omega = \imath_{\varphi^* Y} \varphi^* \omega = \imath_{fY} \omega = f \imath_Y \omega = f \lambda.
\]
Differentiating this identity and using again the fact that $\varphi$ is symplectic, we get
\[
\omega = d( f\lambda) = df\wedge \lambda + f \omega.
\]
Contraction along $Y$ gives
\[
\lambda = \imath_Y (df\wedge \lambda + f \omega) = f \imath_Y \omega = f \lambda.
\]
because $df(Y)=0$ by 0-homogeneity and $\lambda(Y)=\omega(Y,Y)=0$. Therefore, $f=1$ and $\varphi$ satisfies (ii).
\end{proof}

In order to construct homogeneous symplectomorphisms, it is useful to extend also the basic objects of contact geometry to our infinite dimensional setting.

A smooth hypersurface $S$ in $\mathbb{H}$ is said to be of {\em contact type} if it admits a nowhere vanishing smooth 1-form $\alpha$ such that $d\alpha = \omega|_S$ and $\omega$ restricts to a symplectic form on the 2-codimensional subspace $\ker \alpha(x)$, for every $x\in S$. Such a 1-form $\alpha$ is called a {\em contact form} on $S$. It induces a tangent vector field $R_{\alpha}$ on $S$, which is called the {\em Reeb vector field} of $\alpha$ and is defined by the identities
\[
\imath_{R_{\alpha}} d\alpha = 0, \qquad \alpha(R_{\alpha}) = 1.
\]
Therefore, $R_{\alpha}$ is a nowhere vanishing section of the characteristic distribution $\mathscr{D}_S$ of $S$. These are symplectically invariant concepts: It is immediate to check that, if $\varphi$ is a smooth symplectomorphism between open subsets of $\mathbb{H}$, then the 1-form $\varphi^*\alpha$ is a contact form on $\varphi^{-1}(S)$ satisfying $d(\varphi^* \alpha) = \omega|_{\varphi^{-1}(S)}$, and the corresponding Reeb vector field is
\[
R_{\varphi^* \alpha} = \varphi^*(R_{\alpha}).
\]

Let $C$ be an element of $\mathscr{C}$ with smooth boundary. Then $\partial C$ is of contact type with respect to the restriction of the Liouville 1-form to $\partial C$, that is
\[
\alpha := \lambda|_{\partial C}.
\]
Indeed, $d\alpha = d\lambda|_{\partial C} = \omega|_C$, and we have to check that $\omega$ restricts to a symplectic form on $\ker \alpha(x)$, for every $x\in \partial C$.
By (\ref{kerlambda}) there holds
\[
\ker \alpha (x) = T_x \partial C \cap \ker \lambda(x) = \bigl( \R n_{C}(x) \bigr)^{\perp} \cap J \bigl(( \R x )^{\perp}\bigr),
\]
where $n_C(x)$ denotes the unit exterior normal vector to $\partial C$ at $x$. The orthogonal complement of the above space is
\[
\bigl( \ker \alpha (x) \bigr)^{\perp} = \R n_C (x) + (J^T)^{-1} \R x =   \R n_C (x) + J \R x,
\]
and $\omega$ is non-degenerate on this two-dimensional plane because
\[
\omega(n_C(x),Jx) = (J n_C(x),Jx) = (n_C(x),x)>0,
\]
since $C$ is a convex neighborhood of the origin (see (\ref{normale})). By Lemma \ref{fincodim}, $\omega$ is non-degenerate on $\ker \alpha (x)$.

A direct computation shows that the Reeb vector field of $\alpha=\lambda|_{\partial C}$ is
\begin{equation}
\label{reeb}
R_{\alpha}(x) = \frac{2}{(n_C(x),x)} J n_C (x) \qquad \forall x\in \partial C.
\end{equation}
In particular, $R_{\alpha}$ is a positively oriented section of the characteristic distribution of $\partial C$. 

The proof of the following result uses Moser's argument from \cite{mos65}.

\begin{thm}
\label{possohom}
Let $C$ be an element of $\mathscr{C}$ with smooth boundary. Let
\[
\varphi: C \rightarrow \varphi(C) \subset \mathbb{H}
\]
be a smooth symplectomorphism such that the maps $d\varphi$ and $(d\varphi)^{-1}$ are bounded on $C$. Then $\varphi(C)$ is the image  of $C$ by a homogeneous symplectomorphism $\psi$ which is smooth on $\mathbb{H}\setminus \{0\}$. If moreover the third differential of $\varphi$ is bounded on $C$, then the -1-homogeneous maps $d^2 \psi$ and $d^2 \psi^{-1}$ are bounded on the complement of every neighborhood of the origin.
\end{thm}

\begin{proof}
Since $\varphi$ is symplectic, the one-form $\varphi^* \lambda - \lambda$ is closed and hence exact, because $C$ is simply connected: There exists a smooth function $h: C \rightarrow \R$ such that 
\[
\varphi^* \lambda - \lambda = d h.
\]
Since the sets $C$ and $\varphi(C)$ and the map $d\varphi$ are bounded, $dh$ is bounded on $C$. It follows that $h$ is also bounded, because $C$ is bounded and convex. Since $\varphi$ is symplectic, the pull-back of $\lambda|_{\partial \varphi(C)}$ to $\partial C$, that is the one-form
\[
\varphi|_{\partial C}^*( \lambda|_{\partial \varphi(C)} ) = (\lambda + dh)|_{\partial C},
\]
is a contact form on $\partial C$ with differential $\omega|_{\partial C}$. Moreover, the Reeb vector field of this contact form on $\partial C$ is 
\[
\varphi^*(R_{\lambda|_{\partial \varphi(C)}}) = R_{(\lambda+dh)|_{\partial C}}.
\]
Since this vector field is another non-vanishing section of the characteristic distribution of $\partial C$, we have
\begin{equation}
\label{ugu}
R_{(\lambda+dh)|_{\partial C}} = f R_{\lambda|_{\partial C}},
\end{equation}
where $f$ is a nowhere vanishing smooth function on $\partial C$. The fact that $\varphi$ maps the interior of $C$ into the interior of $\varphi(C)$ implies that $\varphi^*(R_{\lambda|_{\partial \varphi(C)}})$ is positively oriented: Indeed, for every $x\in \partial C$ we have
\[
\varphi^*(J n_{\varphi(C)}) (x) =  d\varphi(x)^{-1} J n_{\varphi(C)}(\varphi(x)) = J d\varphi(x)^T n_{\varphi(C)}(\varphi(x)) = g(x) J n_C(x),
\]
for some positive function $g$. Since $R_{\lambda|_{\partial C}}$ is also positively oriented, we conclude that the function $f$ is everywhere positive. 

By formula (\ref{reeb}), the vector fields $R_{\lambda|_{\partial C}}$ and $R_{\lambda|_{\partial \varphi(C)}}$ are bounded and bounded away from zero. Using also the fact that the map $(d\varphi)^{-1}$ is bounded, we obtain that the vector field
\[
R_{(\lambda+dh)|_{\partial C}} = \varphi^*(R_{\lambda|_{\partial \varphi(C)}})
\]
is also bounded and bounded away from zero. Therefore, the positive function $f$ is bounded and bounded away from zero. By applying the 1-form $\lambda+dh$ to (\ref{ugu}), we obtain
\[
1= f \bigr(1+ dh(R_{\lambda|_{\partial C}}) \bigr).
\]
Therefore,
\[
f = \frac{1}{1+dh(R_{\lambda|_{\partial C}})},
\]
and the function $dh(R_{\lambda|_{\partial C}})$ satisfies
\[
-1 < \inf dh(R_{\lambda|_{\partial C}}) \leq \sup dh(R_{\lambda|_{\partial C}}) < +\infty.
\]
We simplify the notation by setting $\alpha:= \lambda|_{\partial C}$ and $R:= R_{\alpha}$. Together with the boundedness of $h$ and $R$, the above bounds imply that the smooth time-dependent tangent vector field
\[
X_t : C \rightarrow TC, \qquad
X_t := - \frac{h|_{\partial C}}{1+t \, dh(R)} R,
\]
is bounded. Since the Hilbert manifold $\partial C$ is complete with respect to the Riemannian structure which is induced by the Hilbert product of $\mathbb{H}$, the flow 
$\{\eta_t:\partial C \rightarrow \partial C\}_{t\in [0,1]}$ of $X_t$, i.e. the solution of
\[
\partial_t \eta_t = X_t(\eta_t), \qquad \eta_0 = \mathrm{id},
\]
is defined for every $t\in [0,1]$. From the identity
\[
(\alpha + t d h )(X_t) = - \frac{h}{1 + t \, dh(R)} - \frac{t\, h}{1 + t \, dh(R)} dh(R) = - h
\]
we find by Cartan's identity
\[
\begin{split}
\frac{d}{dt} \eta_t^* (\alpha + t d h) &= \eta_t^* \bigl( L_{X_t} (\alpha + t d h) + dh \bigr) = \eta_t^* \bigl( \imath_{X_t} d (\alpha + t dh)+ d \bigl((\alpha + t d h)(X_t)\bigr) + d h \bigr) \\ &= \eta^*_t (0 - dh + dh) = 0,
\end{split}
\]
where we have used the fact that $\imath_{X_t} d\alpha =0$, since $X_t$ is parallel to $R$. Together with the condition $\eta_0^* \alpha = \alpha$, the above identity implies that
\[
\eta_t^* (\alpha + t \, dh ) = \alpha
\]
for every $t\in [0,1]$. Then the smooth diffeomorphism 
\[
\tilde{\psi} : \partial C \rightarrow \partial \varphi(C) = \varphi(\partial C), \qquad \tilde{\psi} := \varphi\circ \eta_1,
\]
satisfies
\[
\tilde{\psi}^* ( \lambda|_{\partial \varphi(C)} ) = \eta_1^* ( \varphi^* ( \lambda|_{\partial \varphi(C)} )) = \eta_1^* ( \alpha + dh) = \alpha = \lambda|_{\partial C}.
\]
We claim that the one-homogeneous extension of $\tilde{\psi}$, that is the map
\[
\psi: \mathbb{H} \setminus \{0\} \rightarrow \mathbb{H} \setminus \{0\}, \qquad \psi(rx) = r \tilde{\psi}(x)\quad \forall x\in \partial C, \; \forall r>0,
\]
satisfies $\psi^* \lambda = \lambda$. Indeed, let $x\in \partial C$ and decompose the vector $u\in \mathbb{H}$ as
\[
u = v + \rho Y(x),
\]
where $v\in T_x \partial C$ and $\rho\in \R$. Notice that if $x\in \partial C$ then
\[
d\psi(x) u = d\tilde{\psi}(x) v + \sigma Y(\psi(x)),
\]
for some $\sigma\in \R$, because $\psi$ maps the ray $\R^+ x$ into the ray $\R^+ \psi(x)$.
Then, using the 1-homogeneity of $\lambda$ and the 0-homogeneity of $d\psi$, we have
\[
\begin{split}
(\psi^* \lambda)(rx)&[u] = \lambda(\psi(rx))[ d\psi(rx) u ] = \lambda(r\tilde{\psi}(x))[d\psi(x)u] = r \lambda(\tilde{\psi}(x))[d\tilde{\psi}(x) v + \sigma Y(\psi(x))] \\ & = r \lambda(\tilde{\psi}(x))[d\tilde{\psi}(x) v] = r (\tilde{\psi}^* \lambda)(x)[v] = r \lambda(x)[v] = r \lambda(x)[v+\rho Y(x)] = \lambda(rx)[u].
\end{split}
\]
Therefore, $\psi^* \lambda=\lambda$ and $\psi$ is the required homogeneous symplectomorphism mapping $C$ onto $\varphi(C)$.

If the third differential of $\varphi$ is bounded on $C$, then $h$ has a bounded third differential. It follows that the vector field $X$ has bounded first and second differentials, and hence its time-one map $\eta_1$ and its inverse $\eta_1^{-1}$ have bounded second differentials. By composition the same is true for the maps $\tilde{\psi}$ and $\tilde{\psi}^{-1}$. We conclude that their 1-homogeneous extensions $\psi$ and $\psi^{-1}$ have bounded second differential on $\partial C$ and $\partial \varphi(C)$. Being homogeneous maps of degree -1, $d^2 \psi$ and $d^2 \psi^{-1}$ are bounded on the complement of every neighborhood of the origin.
\end{proof}

\section{The invariance property (vi) and the proof of Theorems \ref{uno} and \ref{due}}
\label{geninv}

We can finally proof the invariance of the symplectic capacity $c_{\mathbb{H}}$ for a class of not necessarily homogeneous symplectomorphisms. This is the precise statement of property (vi) from the Introduction. 

\begin{thm}
\label{invariance2}
Assume that $C\in \widehat{\mathscr{C}}$ has a smooth boundary. Let $\varphi: C \rightarrow \varphi(C) \subset \mathbb{H}$ be a smooth symplectomorphism such that the differentials up to the third order of $\varphi$ and $\varphi^{-1}$ are bounded. If $\varphi(C)$ is a convex neighborhood of the origin then
\[
c_{\mathbb{H}}(\varphi(C)) =  c_{\mathbb{H}}(C).
\]
\end{thm}

\begin{proof} 
By Theorem \ref{possohom} we can find a homogeneous symplectomorphism $\psi: \mathbb{H} \rightarrow \mathbb{H}$ which is smooth on $\mathbb{H}\setminus \{0\}$ and such that $\psi(C)=\varphi(C)$. Moreover, the second differentials of $\psi$ and $\psi^{-1}$ are bounded on the complement of every neighborhood of the origin.

We cannot apply directly Theorem \ref{invariance} to $\psi$ and $C$ because the set $\psi(C)$ might not be in $\widehat{\mathscr{C}}$: Its boundary is certainly smooth, but it might not be strongly convex. We shall overcome this difficulty by an approximation argument.

Since the differential of $H_C$ is Lipschitz-continuous and so is the 0-homogeneous map $d\psi^{-1}$ on the complement of each neighborhood of the origin, the 2-homogeneous function
\[
H_{\psi(C)} = H_C \circ \psi^{-1}
\]
has a Lipschitz-continuous differential. Then the same is true for the 2-homogeneous function
\[
x \mapsto H_{\psi(C)}(x) + \frac{\epsilon}{2} \|x\|^2
\]
which moreover is $\epsilon$-strongly convex, for every $\epsilon>0$. Therefore, the Fenchel conjugate of the above function has a Lipschitz-continuous differential (see e.g.\ \cite[Theorem 18.15]{bc11}) and hence the set
\[
K_{\epsilon} := \set{x\in \mathbb{H}}{ H_{K_{\epsilon}}(x) \leq \frac{1}{2} }, \qquad \mbox{where } 
H_{K_{\epsilon}}(x) := H_{\psi(C)}(x) + \frac{\epsilon}{2} \|x\|^2 \qquad \forall x\in \mathbb{H},
\]
belongs to $\widehat{\mathscr{C}}$. The sets $K_{\epsilon}$ are subsets of $\psi(C)$ and converge to $\psi(C)$ in the Hausdorff metric for $\epsilon\rightarrow 0$. Since $\psi^{-1}$ is Lipschitz-continuous, $\psi^{-1}(K_{\epsilon})$ converges to $C$  in the Hausdorff metric for $\epsilon \rightarrow 0$. Moreover,
\[
H_{\psi^{-1}(K_{\epsilon})} (x) = H_{K_{\epsilon}} \circ \psi (x) = H_{\psi(C)} \circ \psi (x) + \frac{\epsilon}{2} \|\psi(x)\|^2 = H_C(x) + \frac{\epsilon}{2} \|\psi(x)\|^2.
\]
Since $H_C$ is strongly convex and $d^2 \psi$ is bounded on the complement of every neighborhood of the origin, the above function is strongly convex if $\epsilon$ is small enough. Indeed, since this function is 2-homogeneous and smooth on $\mathbb{H}\setminus \{0\}$, it is enough to check that there exists $a>0$ such that
\[
d^2 H_{\psi^{-1}(K_{\epsilon})} (x)[u,u] \geq a  \|u\|^2 \qquad \forall x\in \partial B, \; \forall u\in \mathbb{H},
\]
when $\epsilon$ is small enough. This follows directly from the fact that $H_C$ is smooth and strongly convex, and hence
\[
d^2 H_{C} (x)[u,u] \geq b  \|u\|^2 \qquad \forall x\in \partial B,\; \forall u\in \mathbb{H},
\]
for some $b>0$, and from the bound
\[
d^2 \left( \frac{1}{2} \|\psi\|^2 \right)[u,u] = \|d\psi(x)[u]\|^2 + (\psi(x),d^2 \psi(x)[u,u]) \geq -  c \|u\|^2, \qquad u\in \mathbb{H},
\]
where
\[
c := \sup_{x\in \partial B} \|\psi(x)\| \|d^2 \psi(x)\|.
\]
We conclude that, when $\epsilon$ is small enough, both $\psi^{-1}(K_{\epsilon})$ and $K_{\epsilon}$ belong to $\widehat{\mathscr{C}}$, and Theorem \ref{invariance} implies that
\[
c_{\mathbb{H}}(K_{\epsilon}) = c_{\mathbb{H}}(\psi^{-1}(K_{\epsilon})).
\]
By the continuity of $c_{\mathbb{H}}$ with respect to the Hausdorff metric we find
\[
c_{\mathbb{H}}(\psi(C)) = c_{\mathbb{H}}(C),
\]
and the this follows from the identity $\psi(C)=\varphi(C)$.
\end{proof}

\begin{rem}
Using the continuity of the symplectic capacity $c_{\mathbb{H}}$ with respect to the Hausdorff metric, it should also be possible to  prove the invariance of the symplectic capacity $c_{\mathbb{H}}$ for arbitrary convex sets $C$ and $\varphi(C)$ in $\mathscr{C}$. Indeed, the set $\widehat{\mathscr{C}}$ is dense in $\mathscr{C}$ with respect to the Hausdorff metric: If $C\in \mathscr{C}$, then an approximating set $C_{\epsilon}$ in $\widehat{\mathscr{C}}$ can be defined by setting
\[
H_{C_{\epsilon}} = \left( H_C^* + \frac{\epsilon}{2} \|\cdot\|_*^2 \right)^* + \frac{\epsilon}{2} \|\cdot\|^2,
\]
where $\|\cdot\|_*$ denotes the dual norm on $\mathbb{H}^*$. However in general $\varphi(C_{\epsilon})$ will not be convex even for $\epsilon$ small, and further approximation arguments are needed to complete the proof of the invariance. 
\end{rem}

The existence of a symplectic capacity which satisfies the conditions of Theorems \ref{main} and \ref{invariance2} allows us to prove Theorems \ref{uno} and \ref{due} from the Introduction. 

\begin{proof}[Proof of Theorems \ref{uno} and \ref{due}.]
Since $d^3\varphi$ is bounded and $B_r$ is convex, the convex set $\varphi(B_r)$ is also bounded. Up to composing $\varphi$ with a translation, we may assume that $\varphi(B_r)$ is a neighborhood of $0$. Therefore $\varphi(B_r)$ belongs to $\mathscr{C}$. Since $B_r$ belongs to $\widehat{\mathscr{C}}$ and has smooth boundary, Theorem \ref{invariance2} implies that
\[
c_{\mathbb{H}}(\varphi(B_r)) = c_{\mathbb{H}}(B_r).
\]
Let $P$ be the symplectic projector onto a symplectic closed linear subspace $\mathbb{H}_0$. By the above identity, together with Theorem \ref{main} (ii), (iii) and (iv), we have
\begin{equation}
\label{ingen}
c_{\mathbb{H}_0} (P \varphi(B_r)) \geq c_{\mathbb{H}}(\varphi(B_r)) = c_{\mathbb{H}}(B_r)= r^2  c_{\mathbb{H}}(B) = \pi r^2.
\end{equation}
When $\dim \mathbb{H}_0=2$ the left-hand side of (\ref{ingen}) coincides with with the $\omega$-area of $P\varphi(B_r)$, as we have seen in (\ref{dim2}), and the conclusion of Theorem \ref{uno} follows:
\[
\mathrm{area}_{\omega} (P \varphi(B_r)) \geq  \pi r^2.
\]
When $\dim \mathbb{H}_0=2k$ the left-hand side of (\ref{ingen}) has the bound
\begin{equation}
\label{vite}
c_{\mathbb{H}_0} (P \varphi(B_r))^k \leq \gamma \, \mathrm{vol}_{\omega^k} (P \varphi(B_r)),
\end{equation}
where $\gamma$ is a positive constant which does not depend on $k$. This capacity-volume estimate has been proved by Artstein-Avidan, Milman and Ostrover in \cite{amo08}. From (\ref{ingen}) and (\ref{vite}) we conclude that
\[
\mathrm{vol}_{\omega^k} (P \varphi(B_r)) \geq \gamma^{-1} c_{\mathbb{H}_0} (P \varphi(B_r))^k \geq \gamma^{-1} \pi^k r^{2k}, 
\]
which proves Theorem \ref{due}.
\end{proof}

\section{Proof of Theorem \ref{minicara}}
\label{last}

The aim of this last section is to prove Theorem \ref{minicara} which, as already noticed, follows from standard arguments from Clark duality (see \cite{cla79, cla81, ce80,eke90}).

Let $x: \T \rightarrow \partial C$ be a closed characteristic on $\partial C$. Then $x$ coincides, up to an orientation preserving time reparametrization, with a $T$-periodic solution $y: \R/T\Z \rightarrow \partial C$ of the Hamiltonian equation
\[
-\Omega \dot{y} = dH_C(y),
\]
where $H_C = \mu_C^2/2$ and $\mu_C$ is the Minkowski gauge of $C$. Then $H_C(y)=1/2$ and the symplectic action of $x$ is
\[
\mathbb{A}(x) = \int_{\R/T\Z} y^*(\lambda) = - \frac{1}{2} \int_{\R/T\Z} \langle \Omega \dot{y}, y \rangle\, dt = \frac{1}{2} \int_{\R/T\Z} \langle dH_C(y),y \rangle \, dt =  \int_{\R/T\Z} H_C(y)\, dt = \frac{T}{2},
\]
where we have used the Euler identity for the 2-homogeneous function $H_C$. The continuously differentiable loop
\[
\xi: \T \rightarrow \mathbb{H}^*, \qquad \xi(t) := - \frac{1}{T} \Omega y(Tt),
\]
satisfies
\[
H_{C^0}(\dot{\xi}) = H_{C^0}(- \Omega \dot{y}) =  H_{C^0} (dH_C(y)) = H_C(y) = \frac{1}{2},
\]
where we have used the identity (\ref{leg2}). Therefore, $\dot{\xi}(t)$ belongs to $C^0$ for every $t\in \T$. It follows that
\[
c_{\mathbb{H}}(C) = \frac{1}{4 a_{\infty}(C)} \leq \frac{1}{4 \A^*(\xi)} = \frac{T^2}{4 \A(y)} = \frac{T^2}{4 \A(x)} = \A(x).
\]
This proves the first assertion of Theorem \ref{minicara}.

We now prove the second statement for $p=2$.
Since the function
\[
p\mapsto \|\mu_{C^0}(\dot{\xi})\|_p
\]
is increasing on $[2,+\infty]$, the case $p\in [2,+\infty]$ follows. The case $p\in [1,2)$ requires some concepts from non-smooth analysis as in \cite{cla81} or \cite[Chapter II]{eke90} and will not be presented here, since we do not use this result in this paper. 

Let $\xi:\T \rightarrow \mathbb{H}^*$ be an absolutely continuous curve which maximizes $\mathbb{A}^*$ among all curves $\eta:\T \rightarrow \mathbb{H}^*$
such that $\|\mu_{C^0}(\dot{\eta})\|_2\leq 1$. We must show that $-\Omega^{-1} \xi$ is homothetic to a closed characteristic on $\partial C$ of action $c_{\mathbb{H}}(C)$. Since $\xi$ is a maximizer of $\mathbb{A}^*$ and $\mathbb{A}^*(\xi)>0$, we actually have $\|\mu_{C^0}(\dot{\xi})\|_2=1$, otherwise $\theta \xi$ would still satisfy the constrain for some $\theta>1$, and 
by
\[
\mathbb{A}^*(\theta \xi)=\theta^2 \mathbb{A}^*(\xi) > \A^*(\xi) 
\]
$\xi$ would not be a maximizer. We deduce that $\xi$ maximizes the smooth functional $\mathbb{A}^*$ under the constraint $\Phi_C(\xi)=1/4$, where
\[
\Phi_C(\xi) = \frac{1}{2} \int_{\T} H_{C^0}(\dot{\xi})\, dt = \frac{1}{4} \|\mu_{C^0}(\dot{\xi})\|_2^2.
\]
Since $\Phi_C$ is continuously differentiable on $H^1(\T,\mathbb{H}^*)$, the theorem of Lagrange multipliers implies that
\begin{equation}
\label{criti}
d\mathbb{A}^*(\xi) = \lambda \, d\Phi_C(\xi)
\end{equation}
for some $\lambda\in \R$. By the Euler identity
\[
2 a_2(C) = 2\mathbb{A}^*(\xi) = d\mathbb{A}^*(\xi)[\xi] =  \lambda \, d\Phi_C(\xi)[\xi] = 2 \lambda \, \Phi_C(\xi) = \frac{\lambda}{2},
\]
and hence $\lambda= 4 a_2(C)$. Using the formulas (\ref{difPhi}) and (\ref{difA}), (\ref{criti}) can be rewritten as
\[
\int_{\T} \langle \dot{\eta} ,\Omega^{-1} \xi + 2  a_2(C)  dH_{C^0}(\dot{\xi})\rangle \, dt = 0 \qquad \forall \eta\in H^1(\T,\mathbb{H}^*).
\]
By the Du Bois-Reymond Lemma, there is a constant loop $\bar{y}$ in $\mathbb{H}$ such that
\[
-\Omega^{-1}\xi + \bar{y} = 2 a_2(C) dH_{C^0}(\dot{\xi})  \quad \mbox{a.e. on } \T.
\]
By applying $dH_C$ to both sides we find by (\ref{leg1})
\[
dH_C(-\Omega^{-1}\xi + \bar{y}) = 2 a_2(C) dH_C \circ dH_{C^0} (\dot{\xi}) = 2 a_2(C) \dot{\xi}.
\]
Therefore, the loop $y:=-\Omega^{-1} \xi +\bar{y}:\T \rightarrow \mathbb{H}$ is a 1-periodic solution of the Hamiltonian system
\begin{equation}
\label{hamsi}
-\Omega \dot{y} = \frac{1}{2 a_2(C)} dH_C(y).
\end{equation}
In particular, the function $H_C(y)$ has a constant value $E>0$. It follows that the curve
$x := y/\sqrt{2E}$ satisfies 
\[
H_C(x) = \frac{1}{2E} H_C(y) = \frac{1}{2},
\]
and hence is the required closed characteristic on $\partial C$ homothetic to $y-\bar{y}=-\Omega^{-1} \xi$. The value of its action is by (\ref{hamsi})
\[
\begin{split}
\A(x) &= \frac{1}{2E} \A(y) = -\frac{1}{4E} \int_{\T} \langle \Omega \dot{y},y\rangle \, dt = \frac{1}{8 E a_2(C)}  \int_{\T} \langle dH_C(y),y\rangle \, dt \\ &= \frac{1}{4 E a_2(C)} \int_{\T} H_C(y)\, dt = \frac{1}{4 E a_2(C)} E = \frac{1}{4 a_2(C)} = c_{\mathbb{H}}(C),
\end{split}
\]
as claimed. The proof in the case of a loop $\xi$ which minimizes $\|\mu_{C^0}(\dot{\xi})\|_2$ on the set of absolutely continuous loops with action $\A^*(\xi)=1$ is completely analogous.

We now assume that $\mathbb{H}$ is finite dimensional. The set of $\xi\in H^1(\T,\mathbb{H}^*)$ such that $\|\mu_{C^0}(\dot{\xi})\|_2\leq 1$ projects to a weakly compact subset of the quotient $H^1(\T,\mathbb{H}^*)/\mathbb{H}^*$, where $\mathbb{H}^*$ denotes the subspace of constant loops. The function $\A^*$ is invariant with respect to translations by constants and is weakly continuous in $H^1$, since it is continuous on $H^{1/2}(\T,\mathbb{H}^*)$, which embeds compactly in $H^1(\T,\mathbb{H}^*)$, because $\mathbb{H}^*$ is finite dimensional. Therefore, the supremum which defines $a_2(C)$ is a maximum. Let $\xi$ be a maximizer. As we have seen above, a suitable translated copy $y$ of $-\Omega^{-1}\xi$ satisfies (\ref{hamsi}), from which, applying $H_{C^0}$ and using (\ref{leg2}), we find
\[
H_{C^0}(\dot{\xi})= H_{C^0}(-\Omega \dot{y}) = \frac{1}{4 a_2(C)^2} H_C(y).
\]
Since $H_C(y)$ is constant, so is $H_{C^0}(\dot{\xi})$. It follows that 
\[
\|\mu_{C^0}(\dot{\xi})\|_p = \|\mu_{C^0}(\dot{\xi})\|_2 \qquad \forall p\in [1,+\infty],
\]
so $\xi$ is a maximizer also for the problem which defines $a_p(C)$. The existence of a minimizer for the problem (\ref{infi}) is completely analogous. This concludes the proof of Theorem \ref{minicara}.

\renewcommand{\thesection}{\Alph{section}}
\renewcommand{\theequation}{a.\arabic{equation}}
\setcounter{section}{0}
\setcounter{equation}{0}

\section{Appendix: linear non-squeezing and attractors}

Let $(\mathbb{H},\omega)$ be a symplectic Hilbert space, and let $B$ and $J$ be the unit ball and the complex structure which are determined by a compatible inner product $(\cdot,\cdot)$ on $\mathbb{H}$. Linear symplectomorphisms on $\mathbb{H}$ satisfy the following generalized version of the non-squeezing theorem:

\begin{thm}
\label{linnonsqueez}
Let  $P$ be the symplectic projector onto a $2k$-dimensional symplectic subspace $\mathbb{H}_0$ of $\mathbb{H}$. Then for every linear symplectomorphism $\Phi: \mathbb{H} \rightarrow \mathbb{H}$ there holds
\begin{equation}
\label{app1}
\mathrm{vol}_{\omega^k}(P \Phi(B)) \geq \pi^k.
\end{equation}
The equality holds if and only if the subspace $\Phi^{-1} \mathbb{H}_0$ is $J$-invariant.
\end{thm}

Notice that the above inequality is the one appearing in Theorem \ref{due} (in the case $r=1$, since here we are dealing with linear mappings), but with the sharp constant $\gamma=1$. 

\begin{proof}
First assume that $\mathbb{H}_0$ is $J$-invariant. In this case $P$ is an orthogonal projector, and the inequality (\ref{app1}) is proved in \cite[Theorem 1]{am13} (in the finite-dimensional case, but the proof extends readily to infinite-dimensional Hilbert spaces). There it is also proved that the equality holds if and only if the subspace $\Phi^T \mathbb{H}_0$ is $J$-invariant. Using the fact that $\mathbb{H}_0$ is $J$-invariant and $\Phi$ is symplectic, that is $\Phi^T J \Phi = J$, the latter condition is easily seen to be equivalent to the fact that $\Phi^{-1}\mathbb{H}_0$ is $J$-invariant:
\[
J \Phi^T \mathbb{H}_0 = \Phi^T \mathbb{H}_0 \quad \iff \quad J \Phi^T J \mathbb{H}_0 = \Phi^T J \mathbb{H}_0 \quad \iff \quad \Phi^{-1} \mathbb{H}_0 = J \Phi^{-1} \mathbb{H}_0.
\]
Now we show how the general case can be deduced from the above one. Let $(\cdot,\cdot)'$ be an $\omega$-compatible inner product on $\mathbb{H}$ for which $P$ is an orthogonal projector, and let $B'$ and $J'$ be the corresponding unit ball and complex structure. Let $\Psi: (\mathbb{H},\omega,J') \rightarrow (\mathbb{H},\omega,J)$ be a symplectic and
complex linear isomorphism. It follows that $\Psi$ is an isometry from $(\mathbb{H},(\cdot,\cdot)')$ to $(\mathbb{H},(\cdot,\cdot))$, and hence $\Psi(B')=B$. If we apply the previous case to the symplectic isomorphism $\Phi\Psi$, we obtain
\[
\mathrm{vol}_{\omega^k}(P \Phi(B)) = \mathrm{vol}_{\omega^k}(P \Phi\Psi(B')) \geq \pi^k,
\]
with the equality holding if and only if the subspace $(\Phi\Psi)^{-1} \mathbb{H}_0$ is $J'$-invariant. Using the identity $J' \Psi^{-1} = \Psi^{-1} J$, which follows by inverting $\Psi J' = J \Psi$, we can check that the latter condition is equivalent to the fact that $\Phi^{-1} \mathbb{H}_0$ is $J$-invariant:
\[
\begin{split}
J' (\Phi\Psi)^{-1} \mathbb{H}_0 = (\Phi\Psi)^{-1} \mathbb{H}_0 \quad \iff \quad J' \Psi^{-1}\Phi^{-1} \mathbb{H}_0 = \Psi^{-1}\Phi^{-1} \mathbb{H}_0 \\ \iff \quad \Psi^{-1} J \Phi^{-1} \mathbb{H}_0 = \Psi^{-1}\Phi^{-1} \mathbb{H}_0 \quad \iff \quad J \Phi^{-1} \mathbb{H}_0 = \Phi^{-1} \mathbb{H}_0.
\end{split}
\]
This concludes the proof.
\end{proof}

The above linear result has the following non-linear consequence, which implies that a 1-parameter family of symplectomorphisms cannot have a compact invariant set which is a ``uniform attractor'':

\begin{cor}
Let $\varphi: A \rightarrow A'$ be a $C^1$ symplectomorphism between open subsets of $\mathbb{H}$. Then there cannot exist a compact subset $K\subset A$ with the property
\begin{equation}
\label{attractor}
\varphi(K + rB) \subset K + \theta r B \qquad \forall r\in [0,r_0],
\end{equation}
where $r_0>0$ and $\theta<1$.
\end{cor}

\begin{proof}
When $\dim \mathbb{H}<\infty$, the claim follows from the conservation of volume. Therefore, we may assume that $\mathbb{H}$ is infinite-dimensional.
Assume by contradiction that $\varphi$ satisfies (\ref{attractor}).
Set $\epsilon:=(1-\theta)/3$ and fix some $x_0\in K$. Since $\varphi$ is continuously differentiable, there exists a positive function $r\mapsto \delta(r)$ infinitesimal for $r\rightarrow 0$ such that
\[
\varphi(x_0) + r d\varphi(x_0) (B) \subset \varphi(x_0+rB) + \delta(r) r B.
\]
Fix $r>0$ so that $\delta(r)<\epsilon$ and obtain, using also (\ref{attractor}),
\begin{equation}
\label{app2}
\varphi(x_0) + r d\varphi(x_0) (B) \subset \varphi(x_0+rB) + \epsilon r B \subset K + \left( \theta + \epsilon \right) r B = K + \left( 1-2\epsilon \right) r B.
\end{equation}
Being compact, $K$ can be covered by finitely many balls of radius $\epsilon r$: There exist points $x_1,\dots,x_N\in \mathbb{H}$ such that
\[
K \subset \bigcup_{j=1}^N ( x_j + \epsilon r B).
\]
Together with (\ref{app2}) this implies
\begin{equation}
\label{app3}
\varphi(x_0) + r d\varphi(x_0) (B) \subset \bigcup_{j=1}^N \bigl( x_j + (1-\epsilon)r B \bigr).
\end{equation}
Let $P$ be the orthogonal projector onto a $2k$-dimensional $J$-invariant subspace $\mathbb{H}_0$. By (\ref{app3}) we obtain
\begin{equation}
\label{app4}
\begin{split}
\mathrm{vol}_{\omega^k} \bigl( P \varphi(x_0) + r P d\varphi(x_0) (B) \bigr) &\leq \mathrm{vol}_{\omega^k} \Bigl( \bigcup_{j=1}^N \bigl( P x_j + (1-\epsilon)r (B\cap \mathbb{H}_0) \bigr) \Bigr) \\ &\leq N \mathrm{vol}_{\omega^k} \bigl( (1-\epsilon)r (B\cap \mathbb{H}_0) \bigr) = N \pi^k (1-\epsilon)^{2k} r^{2k}.
\end{split}
\end{equation}
On the other hand, by Theorem \ref{linnonsqueez} we have
\begin{equation}
\label{app5}
\mathrm{vol}_{\omega^k} \bigl( P \varphi(x_0) + r P d\varphi(x_0) B \bigr) = r^{2k} \mathrm{vol}_{\omega^k} \bigl( P d\varphi(x_0) B \bigr) \geq \pi^k r^{2k}.
\end{equation}
By (\ref{app4}) and (\ref{app5}) we obtain
\[
1 \leq N (1-\epsilon)^{2k}.
\]
Since the right-hand side is infinitesimal for $k\rightarrow \infty$, we find a contradiction which proves our assertion.
\end{proof}


\newcommand{\etalchar}[1]{$^{#1}$}
\providecommand{\bysame}{\leavevmode\hbox to3em{\hrulefill}\thinspace}
\providecommand{\MR}{\relax\ifhmode\unskip\space\fi MR }
\providecommand{\MRhref}[2]{%
 \href{http://www.ams.org/mathscinet-getitem?mr=#1}{#2}
}
\providecommand{\href}[2]{#2}

\end{document}